\documentclass[12pt]{amsart}
\usepackage{amsmath,amssymb}
\usepackage{bbm}
\usepackage{graphicx,tikz,mathtools}
\usepackage{xcolor}
\usepackage[left=2.65cm,right=2.65cm,top=2.7cm,bottom=2.7cm]{geometry}
\usepackage[pdftex]{hyperref}
\usepackage{cite}
\usepackage{mathrsfs} 
\usepackage{relsize}
\usepackage{caption}
\hypersetup{
	colorlinks=true,
	linkcolor=blue, 
	citecolor=blue,
	filecolor=blue,
	urlcolor=blue,
}

\makeatletter
\@namedef{subjclassname@2020}{\textup{2020} Mathematics Subject Classification}
\makeatother

\newtheorem{theorem}{Theorem}[section]
\newtheorem{proposition}[theorem]{Proposition}
\newtheorem{corollary}[theorem]{Corollary}
\newtheorem{lemma}[theorem]{Lemma}

\theoremstyle{definition}

\numberwithin{equation}{section}

\newcommand{\e}{\epsilon}
\newcommand{\R}{\mathbb{R}}

\newcommand{\N}{\mathbb{N}}
\renewcommand\leq\leqslant
\renewcommand\geq\geqslant
\renewcommand\le\leqslant
\renewcommand\ge\geqslant

\renewcommand*\textcircled[1]{\tikz[baseline=(char.base)]{
  \node[shape=circle,draw,inner sep=1.5pt] (char) {#1};}}

\allowdisplaybreaks

\begin{document}
 
\title[A fractional critical problem in the halfspace]{A fractional critical problem in the halfspace \\ with Neumann conditions}

\author[A. DelaTorre]{Azahara DelaTorre}
\address{Dipartimento di Matematica ``Guido Castelnuovo'', Sapienza Universit\`a di Roma, Piazzale Aldo Moro 5, 00185 Roma, Italy}
\email{azahara.delatorrepedraza@uniroma1.it}

\author[S. Dipierro]{Serena Dipierro}
\address{Department of Mathematics and Statistics, The University of Western Australia, 35 Stirling Highway, Crawley WA 6009, Australia}
\email{serena.dipierro@uwa.edu.au}

\author[A. Pistoia]{Angela Pistoia}
\address{Dipartimento di Scienze di Base e Applicate per l'Ingegneria, Sapienza Universit\`a di Roma, Via Antonio Scarpa 16, 00161 Roma, Italy}
\email{angela.pistoia@uniroma1.it}

\author[E. Valdinoci]{Enrico Valdinoci}
\address{Department of Mathematics and Statistics, The University of Western Australia, 35 Stirling Highway, Crawley WA 6009, Australia}
\email{enrico.valdinoci@uwa.edu.au}

\keywords{Fractional Laplacian, critical exponent, nonlocal Neumann condition.}

\thanks{The first author acknowledges financial support from the Spanish Ministry of Science and Innovation (MICINN), through the IMAG-Maria de Maeztu Excellence Grant CEX2020-001105-M/AEI/10.13039/501100011033. She is also supported by the MICIN/AEI Grants PID2023-150166NB-I00 and PID2024-155314NB-I00 and by J. Andalucia (FQM-116); RED2022-134784-T, funded by MCIN/AEI/10.13039/501100011033; Fondi Ateneo - Sapienza Università di Roma and INdAM-GNAMPA Projects 2025 and 2026 with codes CUP E5324001950001 and CUP E53C25002010001. The second author is supported by the Australian Future Fellowship FT230100333. The fourth author is supported by the Australian Laureate Fellowship FL190100081. This article was written on the occasion of a pleasant visit by the second and fourth authors to Sapienza Università di Roma in July 2026. It is a pleasure to thank Sapienza Università di Roma for its support and welcoming environment.}

\subjclass[2020]{35J10, 35R11}

\begin{abstract}
Given~$s\in(0,1)$ and~$n>2s$, we prove that the critical Sobolev quotient associated with the fractional Neumann problem in the halfspace admits a minimizer. As a consequence, after a suitable normalization, we obtain a nontrivial weak solution of
\begin{equation*}
\begin{cases} (-\Delta)^s u=u^{2^*_s-1} &{\mbox{ in }}\R^n_+,\\
{\mathcal{N}}_s u=0&{\mbox{ in }}\R^n_-,
\end{cases}
\end{equation*}
where~${\mathcal{N}}_s$ represents the nonlocal Neumann condition of exterior type.
\end{abstract}

\maketitle
 
\section{Introduction} \label{S.Introduction}
\subsection{Mathematical setting and main results}
The goal of this paper is to study a fractional problem with critical exponent in the halfspace under nonlocal Neumann conditions. The main difficulty is that the Neumann condition is genuinely nonlocal: it is prescribed on the whole complement of the halfspace and couples the values of the solution inside and outside the domain. Consequently, the classical restriction and reflection arguments available for local Neumann problems do not apply directly. We refer to~\cite{MR3651008,MR4102340,MR4630055,MR5023029,2025arXiv250800337B,2025arXiv251013340D} for several aspects of nonlocal Neumann problems and related models.

Specifically, given~$s\in(0,1)$ and~$n>2s$, we define the fractional critical exponent~$2^*_s:=\frac{2n}{n-2s}$ and we want to find nontrivial solutions of the equation
\begin{equation}\label{pb1}
\begin{cases} (-\Delta)^s u=u^{2^*_s-1} &{\mbox{ in }}\R^n_+,\\
{\mathcal{N}}_s u=0&{\mbox{ in }}\R^n_-.
\end{cases}
\end{equation}
Here, we are using the standard notation~$\R^n_+:=\R^{n-1}\times(0,+\infty)$,
$\R^n_-:=\R^{n-1}\times(-\infty,0)$, and
$$ (-\Delta)^s u(x):=c_{n,s}\int_{\R^n}\frac{2u(x)-u(x+y)-u(x-y)}{|y|^{n+2s}}\,dy,$$
where~$c_{n,s}$ is a suitable dimensional constant (its exact value will play no role in this paper).

The nonlocal Neumann condition in~$\R^n_-$ (see~\cite{MR3651008})
is given by
$$ {\mathcal{N}}_s u(x):=c_{n,s}\int_{\R^n_+}\frac{u(x)-u(y)}{|x-y|^{n+2s}}\,dy.$$

To find a solution of~\eqref{pb1} we use a variational argument. For this, we let
$$ Q(\R^n_+):=(\R^n_+\times\R^n_+)\cup(\R^n_-\times\R^n_+)\cup(\R^n_+\times\R^n_-)$$
and define
\begin{equation}\label{SDEF} {\mathfrak{S}}_s:=\inf_{{{\|u\|_{L^{2^*_s}(\R^n_+)}=1}}}
\frac{c_{n,s}}{2}\iint_{Q(\R^{n}_+)}\frac{(u(x)-u(y))^2}{|x-y|^{n+2s}}\,dx\,dy.
\end{equation}
The main result of this paper reads as follows.

\begin{theorem}\label{EXIST:GRS} The infimum in~\eqref{SDEF} is attained. 
\end{theorem}

\begin{corollary}\label{0sqpjdlwbmlgtmlno2CO}
There exists a nontrivial weak solution~$u$ of~\eqref{pb1}.
\end{corollary}

It is also useful to compare the infimum in~\eqref{SDEF}
with the corresponding quantity in the whole of~$\R^n$. More specifically, we define
\begin{equation}\label{CCLSDMN}\begin{split} {\mathfrak{S}}_s^\star&:=\inf_{{\|u\|_{L^{2^*_s}(\R^n)}=1}}
\frac{c_{n,s}}{2}\iint_{\R^{2n}}\frac{(u(x)-u(y))^2}{|x-y|^{n+2s}}\,dx\,dy
.\end{split}\end{equation}

The key ingredient in the proof of Theorem~\ref{EXIST:GRS} is a strict comparison between the halfspace and whole-space critical levels.

\begin{theorem}\label{INterM0}
We have that
\begin{equation}\label{EWNS:02oe-034.P}
{\mathfrak{S}}_s^\star>{\mathfrak{S}}_s.
\end{equation}
\end{theorem}

The existence statement in Corollary~\ref{0sqpjdlwbmlgtmlno2CO} and the strict comparison in Theorem~\ref{INterM0} were obtained independently and contemporaneously in~\cite[Theorems~1.2 and~1.3]{cinti2026existencenonexistenceresultsfractional}, by a different approach. Prior to that, Musina and Nazarov~\cite[Theorem~1.3]{MR4193437} proved that the strict inequality between the halfspace and whole-space fractional Sobolev constants implies attainment of the halfspace infimum. Their work then identified regimes in which the strict inequality can be verified, in particular when~$n=1$ and~$s$ is sufficiently close to~$\left(\frac12\right)^-$, or when~$n\ge2$ and~$s$ is sufficiently close to~$1^-$,
see~\cite[Theorem~4.7]{MR4193437}. Instead,
Theorem~\ref{INterM0} here allows us to establish the attainment of the minimizer for every~$s\in(0,1)$ and~$n>2s$.

It is instructive to compare this problem with its local counterpart. In the classical critical case, the whole-space Sobolev extremals are the Aubin--Talenti bubbles~\cite{MR448404,MR463908}. By centering such a bubble on the boundary hyperplane and restricting it to the halfspace, in the local case one immediately obtains a solution of the corresponding local Neumann problem. This simple construction is no longer available in the present nonlocal setting: restricting a whole-space fractional bubble to the halfspace does not, in general, satisfy the exterior nonlocal Neumann condition. A genuinely nonlocal variational argument is therefore required.

\subsection{Future directions of investigation}
A natural continuation of the present work is the corresponding critical Neumann problem on a bounded smooth domain~$\Omega\subset\R^n$. The local critical Neumann problem has a substantial literature, beginning with early variational existence results of Adimurthi--Yadava and Adimurthi--Mancini and followed by the analysis of least-energy solutions, boundary concentration, and the role of the geometry of~$\partial\Omega$; see~\cite{AdimurthiYadava1990,AdimurthiMancini1991,NiPanTakagi1992,AdimurthiPacellaYadava1993,AdimurthiMancini1994,AdimurthiManciniYadava1995,AdimurthiPacellaYadava1995}. In particular, the large-parameter analysis shows that least-energy solutions concentrate at the boundary and that the mean curvature governs the location of concentration points~\cite{AdimurthiPacellaYadava1993,AdimurthiManciniYadava1995,AdimurthiPacellaYadava1995}. Motivated by this classical picture, in a forthcoming paper we will consider, for~$\lambda>0$,
\begin{equation}\label{bounded-problem}
\begin{cases}
(-\Delta)^s u+\lambda u=u^{2^*_s-1} & \text{in }\Omega,\\
\mathcal N_su=0 & \text{in }\R^n\setminus\Omega.
\end{cases}
\end{equation}
The associated critical level is
\begin{equation}\label{bounded-level}
\mathfrak S_{s,\lambda}(\Omega):=
\inf_{u\not\equiv0}
\frac{\displaystyle \frac{c_{n,s}}2\iint_{Q(\Omega)}
\frac{(u(x)-u(y))^2}{|x-y|^{n+2s}}\,dx\,dy
+\lambda\int_\Omega u^2\,dx}
{\displaystyle \|u\|_{L^{2^*_s}(\Omega)}^2},
\end{equation}
where
\[
Q(\Omega):=\R^{2n}\setminus\big((\R^n\setminus\Omega)\times(\R^n\setminus\Omega)\big).
\]

We expect that the halfspace level obtained in the present paper is precisely the compactness threshold for this bounded-domain problem. More specifically, a natural goal is to prove that
the condition
\begin{equation}\label{bounded-threshold}
\mathfrak S_{s,\lambda}(\Omega)<\mathfrak S_s
\end{equation}
implies that the infimum in~\eqref{bounded-level} is attained. Indeed, boundary concentration should produce, after blow-up, a profile in the halfspace and therefore carry at least the energy~$\mathfrak S_s$, whereas an interior concentration carries the larger whole-space energy~$\mathfrak S_s^\star>\mathfrak S_s$. Thus, the strict inequality~\eqref{bounded-threshold} should rule out loss of compactness and yield a minimizer, which, after a suitable normalization, provides a nontrivial solution of~\eqref{bounded-problem}. In the spirit of the classical large-$\lambda$ theory~\cite{NiPanTakagi1992,AdimurthiPacellaYadava1993,AdimurthiMancini1994,AdimurthiManciniYadava1995,AdimurthiPacellaYadava1995}, the main further issue is then to determine geometric conditions on~$\Omega$ under which~\eqref{bounded-threshold} holds for large~$\lambda$.

A natural strategy for proving~\eqref{bounded-threshold} is to translate and rescale a halfspace minimizer near a point of~$\partial\Omega$ and use it as a test function in~\eqref{bounded-level}. The first correction to the halfspace energy should then encode the geometry of the domain, in a way which is expected to depend substantially on~$s$. Thus, a detailed understanding of the halfspace ground state is expected to be an essential ingredient in determining whether~$\mathfrak S_{s,\lambda}(\Omega)<\mathfrak S_s$ for large~$\lambda$, and in identifying the boundary points where least-energy solutions may concentrate.
The geometric nature of the first correction is also expected to depend essentially on~$s$
and be related to either the classical or nonlocal mean curvature.
\medskip

For other critical problems in the fractional setting, see e.g.~\cite{MR3060890,MR3089742,MR3271254,MR3379042,MR3635980,MR3766988,MR3914955} and the references therein.

\subsection{Organization of the paper}
The paper is organized as follows. Section~\ref{EXTNE}
considers an extension operator induced by the nonlocal Neumann condition and provides a general bound in Lebesgue spaces.
In Section~\ref{INterM0sec} we prove Theorem~\ref{INterM0}, which provides the strict energy gap used in the compactness argument. Section~\ref{EXIST:GRS:A} is devoted to the proof of Theorem~\ref{EXIST:GRS}, while Corollary~\ref{0sqpjdlwbmlgtmlno2CO} is proved in Section~\ref{0sqpjdlwbmlgtmlno2COse}. Some auxiliary convexity, cutoff, kernel, density, and Sobolev estimates are collected in the appendix.

\section{A useful bound for the nonlocal Neumann condition}\label{EXTNE}

This is a general estimate on an extension operator
induced by the nonlocal Neumann condition.

\begin{lemma}\label{NKDF:HO} Let~$u:\R^n_+\to\R$ and, for all~$x\in\R^n_-$, let
$$ T_u(x):=\frac{\displaystyle\int_{\R^n_+}\frac{u(z)}{|x-z|^{n+2s}}\,dz}{\displaystyle\int_{\R^n_+}\frac{dz}{|x-z|^{n+2s}}}.$$
Then, for all~$p\in(1,+\infty]$, 
\begin{equation}\label{0pkjdmwr034iyhk92IKStuMD.3} \|T_u\|_{L^p(\R^n_-)}\le C\,\|u\|_{L^p(\R^n_+)},\end{equation}
for some~$C>0$ depending only on~$n$, $s$, and~$p$.
\end{lemma}

\begin{proof} It is immediate that~$ \|T_u\|_{L^\infty(\R^n_-)}\le\|u\|_{L^\infty(\R^n_+)}$,
hence we can focus on the case~$p\in(1,+\infty)$.

Given~$\theta>0$ and~$y'\in\R^{n-1}$, we define
$$ K_\theta(y'):=\frac{1}{\big(|y'|^2+\theta^2\big)^{\frac{n+2s}2}}$$
and observe that, by scaling,
\begin{equation}\label{FGHANS:02}
\|K_\theta\|_{L^1(\R^{n-1})}=\int_{\R^{n-1}}
\frac{dy'}{\big(|y'|^2+\theta^2\big)^{\frac{n+2s}2}}=\frac{C^\star_{n,s}}{\theta^{1+2s}},
\end{equation}for some positive constant~$C^\star_{n,s}$.

Also, given~$z'\in\R^{n-1}$ and~$\zeta\in(0,+\infty)$, we use the notation~$u^{[\zeta]}(z'):=u(z',\zeta)$.
Hence, for all~$p\in[1,+\infty]$, a standard inequality for convolutions (see e.g.~\cite[Theorem~9.1]{MR3381284}) yields that
\begin{equation}\label{FGHANS:01}
\| u^{[\zeta]}*K_\theta\|_{L^p(\R^{n-1})}\le \| u^{[\zeta]}\|_{L^p(\R^{n-1})} \|K_\theta\|_{L^1(\R^{n-1})}.\end{equation}

Now we let
\begin{eqnarray*} S_u(x)&:=&\int_{\R^n_+} \frac{u(z)}{|x-z|^{n+2s}}\,dz\\&=&
\int_0^{+\infty} \left(
\int_{\R^{n-1}} \frac{u(z',\zeta)}{\big(|x'-z'|^2+|x_n-\zeta|^2\big)^{\frac{n+2s}2}}\,dz'
\right)\,d\zeta\\&=&
\int_0^{+\infty} 
u^{[\zeta]}*K_{|x_n-\zeta|}(x')\,d\zeta.
\end{eqnarray*}
Accordingly, for every~$\eta\in(0,+\infty)$,
\begin{eqnarray*} \|S_u^{[-\eta]}\|_{L^{p}(\R^{n-1})}&=&
\left(\int_{\R^{n-1}} \left| \int_0^{+\infty} u^{[\zeta]}*K_{\eta+\zeta}(w')\,d\zeta\right|^p\,dw' \right)^{\frac1p}
\\&\le&\int_0^{+\infty}
\left(\int_{\R^{n-1}} \left|  u^{[\zeta]}*K_{\eta+\zeta}(w')\right|^p\,dw' \right)^{\frac1p}
\,d\zeta,
\end{eqnarray*}
where the last step is a consequence of the Minkowski Integral Inequality (see e.g.~\cite[Theorem~A.1]{MR4999446}).

Consequently, by~\eqref{FGHANS:01},
\begin{equation*}
\|S_u^{[-\eta]}\|_{L^{p}(\R^{n-1})}\le \int_0^{+\infty} \|  u^{[\zeta]}*K_{\eta+\zeta}\|_{L^p(\R^{n-1})} \,d\zeta
\le \int_0^{+\infty} \| u^{[\zeta]}\|_{L^p(\R^{n-1}) }\|K_{\eta+\zeta}\|_{L^1(\R^{n-1})}\,d\zeta.\end{equation*}
{F}rom this and~\eqref{FGHANS:02} it follows that
\begin{equation*}
\|S_u^{[-\eta]}\|_{L^{p}(\R^{n-1})}\le C^\star_{n,s}
\int_0^{+\infty} \frac{ \| u^{[\zeta]}\|_{L^p(\R^{n-1})} }{ (\eta+\zeta)^{1+2s}} \,d\zeta
=C^\star_{n,s} \int_0^{+\infty}
\frac{ U(\zeta) }{ (\eta+\zeta)^{1+2s}} \,d\zeta
,\end{equation*}
where~$U(\zeta):=\| u^{[\zeta]}\|_{L^p(\R^{n-1})}$.

Thus, employing the change of variable~$\sigma:=\frac\zeta\eta$ and
using again the Minkowski Integral Inequality,
\begin{equation}\label{0pkjdmwr034iyhk92IKStuMD}
\begin{split}
\left(\int_0^{+\infty} \eta^{2sp}\|S_u^{[-\eta]}\|_{L^{p}(\R^{n-1})}^p\,d\eta\right)^{\frac1p}
&\le C^\star_{n,s}\left( \int_0^{+\infty}\left|
\int_0^{+\infty} \frac{\eta^{2s} U(\zeta) }{ (\eta+\zeta)^{1+2s}} \,d\zeta\right|^p\,d\eta\right)^{\frac1p}
\\& =C^\star_{n,s}\left( \int_0^{+\infty}\left|
\int_0^{+\infty} \frac{ U(\eta\sigma) }{ (1+\sigma)^{1+2s}} \,d\sigma\right|^p\,d\eta\right)^{\frac1p}
\\&\le C^\star_{n,s} \int_0^{+\infty} \left( \int_0^{+\infty}
\frac{ |U(\eta\sigma)|^p }{(1+\sigma)^{(1+2s)p}}\,d\eta\right)^{\frac1p} \,d\sigma
,\end{split}\end{equation}
for some~$C>0$, depending only on~$n$, $s$, and~$p$.

It is also useful to observe that, substituting for~$x_n:=\eta\sigma$,
\begin{eqnarray*}&&
\int_0^{+\infty} |U(\eta\sigma)|^p \,d\eta=
\int_0^{+\infty} \| u^{[\eta\sigma]}\|_{L^p(\R^{n-1})}^p\,d\eta=
\int_{\R^n_+} |u(x',\eta\sigma)|^p \,dx'\,d\eta\\
&&\qquad=\frac1\sigma\int_{\R^n_+} |u(x)|^p \,dx=\frac{\|u\|_{L^p(\R^n_+)}^p}{\sigma}.
\end{eqnarray*}
Plugging this information into~\eqref{0pkjdmwr034iyhk92IKStuMD} we arrive at
\begin{equation}\label{0pkjdmwr034iyhk92IKStuMD.2}
\begin{split}
\left(\int_0^{+\infty} \eta^{2sp}\|S_u^{[-\eta]}\|_{L^{p}(\R^{n-1})}^p\,d\eta\right)^{\frac1p}
&\le C^\star_{n,s}\,\|u\|_{L^p(\R^n_+)} \int_0^{+\infty} \frac{d\sigma}{\sigma^{\frac1p}(1+\sigma)^{1+2s}}\\
&=\widetilde{C}_{n,s,p}\,\|u\|_{L^p(\R^n_+)}
,\end{split}\end{equation}
for some constant~$\widetilde{C}_{n,s,p}>0$ (here we are using that~$p>1$ to ensure the convergence at the origin of the above integral).

We also remark that
\begin{eqnarray*}&&
\|T_u\|_{L^p(\R^n_-)}^p=\mathlarger{\mathlarger {\mathlarger{\mathlarger\int}}}_{\R^n_-}\left|
\frac{S_u(x)}{\displaystyle\int_{\R^n_+}\frac{dz}{|x-z|^{n+2s}}}
\right|^p\,dx=\widehat{C}_{n,s}\int_{\R^n_-} |x_n|^{2sp}|S_u(x)|^p\,dx\\
&&\quad=\widehat{C}_{n,s}\int_0^{+\infty}\left(
\int_{\R^{n-1}}\eta^{2sp} |S_u(x',-\eta)|^p\,dx'
\right)\,d\eta=\widehat{C}_{n,s}
\int_0^{+\infty} \eta^{2sp}\|S_u^{[-\eta]}\|_{L^{p}(\R^{n-1})}^p\,d\eta,
\end{eqnarray*}
for some~$\widehat{C}_{n,s}>0$.

This and~\eqref{0pkjdmwr034iyhk92IKStuMD.2} yield~\eqref{0pkjdmwr034iyhk92IKStuMD.3}, as desired.
\end{proof}

We stress that Lemma~\ref{NKDF:HO} does not hold when~$p=1$: indeed, if~$u\ge0$
and~$u\ge1$ in~$(0,1)^n$, the substitution~$y:=\left(\frac{z-x}{z_n-x_n},-x_n\right)$ yields that, for some~$C>0$, possibly varying from line to line,
\begin{eqnarray*}
 \|T_u\|_{L^1(\R^n_-)}&=&C\int_{\R^n_-} \left( \int_{\R^n_+}\frac{|x_n|^{2s} \, u(z)}{|x-z|^{n+2s}}\,dz
 \right)\,dx\\&=&
 C\int_{\R^n_+} \left( \int_{\R^n_-}\frac{|x_n|^{2s} }{|x-z|^{n+2s}}\,dx
 \right) \,u(z)\,dz\\&=& C\int_{\R^n_+} \left( \int_{\R^n_-}\frac{|x_n|^{2s} }{\big(|x'-z'|^2+(z_n-x_n)^2\big)^{\frac{n+2s}2}}\,dx
 \right) \,u(z)\,dz\\&=& C\int_{\R^n_+} \left( \int_{\R^n_+}\frac{y_n^{2s} }{
(z_n+y_n)^{1+2s} \big(|y'|^2+1\big)^{\frac{n+2s}2}}\,dy
 \right) \,u(z)\,dz\\&=& C\int_{\R^n_+} \left( \int_0^{+\infty}\frac{y_n^{2s} }{
(z_n+y_n)^{1+2s} }\,dy_n
 \right) \,u(z)\,dz\\&\ge& C\int_{(0,1)^n} \left( \int_0^{+\infty}\frac{y_n^{2s} }{
(z_n+y_n)^{1+2s} }\,dy_n
 \right) \,dz\\&\ge&  C\int_0^{+\infty}\frac{y_n^{2s} }{
(1+y_n)^{1+2s} }\,dy_n\\&=&+\infty.
\end{eqnarray*}

\section{Proof of Theorem~\ref{INterM0}}\label{INterM0sec}
The proof of Theorem~\ref{INterM0} is based on an explicit energy comparison
with the known minimizers of~\eqref{CCLSDMN}. 

\begin{proof}[Proof of Theorem~\ref{INterM0}] We know that the minimum in \eqref{CCLSDMN}
is attained (see~\cite{MR717827}) by a radially symmetric function~$\omega_s$
with~$\|\omega_s\|_{L^{2^*_s}(\R^n)}^2=1$, given explicitly by
$$ \omega_s(x)=\frac{\alpha_{n,s}}{(1+|x|^2)^{\frac{n-2s}2}},$$
for a suitable normalizing constant~$\alpha_{n,s}>0$.

Given~$d>0$, to be taken conveniently large in what follows, possibly in dependence of~$n$ and~$s$, we let
$$\omega_{s,d}(x',x_n):=\omega_s(x',x_n-d).$$
We observe that if~$x_n<-d$ then
$$1+|x|^2=1+|x'|^2+x_n^2\ge 1+|x'|^2+\frac{x_n^2}2+\frac{d^2}2\ge
\frac{d^2+|x|^2}2$$
and therefore, using the change of variable~$X:=\frac{x}{d}$,
\begin{equation*}\begin{split}&
1-\| \omega_s \|_{L^{2^*_s}(\R^n_+)}^{2^*_s}
=\int_{\R^{n-1}\times(-\infty,-d)} \omega_s^{2^*_s}(x)\,dx
=\alpha_{n,s}^{2^*_s}\int_{\R^{n-1}\times(-\infty,-d)} \frac{dx}{(1+|x|^2)^n}\\&\quad
\le2^n\alpha_{n,s}^{2^*_s}\int_{\R^{n-1}\times(-\infty,-d)} \frac{dx}{(d^2+|x|^2)^n}=
\frac{2^n\alpha_{n,s}^{2^*_s}}{d^n}\int_{\R^{n-1}\times(-\infty,-1)} \frac{dX}{(1+|X|^2)^n}\le \frac{C_{n,s}}{d^n},
\end{split}
\end{equation*} for some~$C_{n,s}\ge 1$, which we take as possibly varying from line to line in what follows.

As a result, assuming that~$d$ is large enough,
\begin{equation}\label{g.SK02w-norm}
\| \omega_s \|_{L^{2^*_s}(\R^n_+)}^{2}\ge \left(1-\frac{C_{n,s}}{d^n}\right)^{\frac{2}{2^*_s}}\ge  1-\frac{C_{n,s}}{d^n}.
\end{equation}

Furthermore, if~$e_n:=(0,\dots,0,1)$, $X\in B_1(-2e_n)$ and~$Y\in B_1(-5e_n)$, then~$|X|\le3$ and~$|Y|\ge4$, therefore
\begin{eqnarray*}
&&\frac{1}{(1+d^2|X|^2)^{\frac{n-2s}2}}-\frac{1}{(1+d^2|Y|^2)^{\frac{n-2s}2}}\ge
\frac{1}{(1+9d^2)^{\frac{n-2s}2}}-\frac{1}{(1+16d^2)^{\frac{n-2s}2}}\\&&\quad=
\frac1{d^{n-2s}}\left(\frac{1}{(d^{-2}+9)^{\frac{n-2s}2}}-\frac{1}{(d^{-2}+16)^{\frac{n-2s}2}}\right)\ge\frac1{2d^{n-2s}}\left(\frac{1}{9^{\frac{n-2s}2}}-\frac{1}{16^{\frac{n-2s}2}}\right),
\end{eqnarray*}
as long as~$d$ is large enough.

As a result,
\begin{equation}\label{g.SK02w2}\begin{split}&
\frac1{\alpha_{n,s}^2}\iint_{(\R^{n-1}\times(-\infty,-d))^2}\frac{(\omega_{s}(x)-\omega_{s}(y))^2}{|x-y|^{n+2s}}\,dx\,dy\\&\quad=
\iint_{(\R^{n-1}\times(-\infty,-d))^2}
\left(\frac{1}{(1+|x|^2)^{\frac{n-2s}2}}-\frac{1}{(1+|y|^2)^{\frac{n-2s}2}}\right)^2
\frac{dx\,dy}{|x-y|^{n+2s}}\\&\quad={d^{n-2s}}
\iint_{(\R^{n-1}\times(-\infty,-1))^2}
\left(\frac{1}{(1+d^2|X|^2)^{\frac{n-2s}2}}-\frac{1}{(1+d^2|Y|^2)^{\frac{n-2s}2}}\right)^2
\frac{dX\,dY}{|X-Y|^{n+2s}}\\&\quad\ge{d^{n-2s}}
\iint_{B_1(-2e_n)\times B_1(-5e_n)}
\left(\frac{1}{(1+d^2|X|^2)^{\frac{n-2s}2}}-\frac{1}{(1+d^2|Y|^2)^{\frac{n-2s}2}}\right)^2
\frac{dX\,dY}{|X-Y|^{n+2s}}\\&\ge\frac{1}{C_{n,s}\,d^{n-2s}}.
\end{split}
\end{equation}

In addition,
\begin{eqnarray*}&&\!\!\!\!\!\!\!\!\!\!\!\!\!\!\!
{\mathfrak{S}}_s\,
\|\omega_{s,d}\|_{L^{2^*_s}(\R^n_+)}^2\\&\le&
\frac{c_{n,s}}{2}\iint_{Q(\R^{n}_+)}\frac{(\omega_{s,d}(x)-\omega_{s,d}(y))^2}{|x-y|^{n+2s}}\,dx\,dy\\&=&
\frac{c_{n,s}}{2}\iint_{\R^{2n}}\frac{(\omega_{s,d}(x)-\omega_{s,d}(y))^2}{|x-y|^{n+2s}}\,dx\,dy-
\frac{c_{n,s}}{2}\iint_{\R^n_-\times\R^n_-}\frac{(\omega_{s,d}(x)-\omega_{s,d}(y))^2}{|x-y|^{n+2s}}\,dx\,dy\\&=&{\mathfrak{S}}_s^\star
-\frac{c_{n,s}}{2}\iint_{(\R^{n-1}\times(-\infty,-d))^2}\frac{(\omega_{s}(x)-\omega_{s}(y))^2}{|x-y|^{n+2s}}\,dx\,dy.
\end{eqnarray*}
This, together with~\eqref{g.SK02w-norm} and~\eqref{g.SK02w2}, yields that
\begin{eqnarray*}
{\mathfrak{S}}_s^\star&\ge&{\mathfrak{S}}_s\,
\|\omega_{s,d}\|_{L^{2^*_s}(\R^n_+)}^2+\frac{c_{n,s}}{2}\iint_{(\R^{n-1}\times(-\infty,-d))^2}\frac{(\omega_{s}(x)-\omega_{s}(y))^2}{|x-y|^{n+2s}}\,dx\,dy\\&\ge&
{\mathfrak{S}}_s\left(1-\frac{C_{n,s}}{d^n}\right)+\frac{1}{C_{n,s}\,d^{n-2s}},
\end{eqnarray*} 
which is strictly larger than~${\mathfrak{S}}_s$ provided that~$d$ is chosen so large
that~${C_{n,s}\,{\mathfrak{S}}_s}<\frac{d^{2s}}{C_{n,s}}$.
\end{proof}

\section{Proof of Theorem~\ref{EXIST:GRS}}\label{EXIST:GRS:A}

This section is devoted to the proof of Theorem~\ref{EXIST:GRS}. 
We borrow some techniques developed in~\cite{MARIA} for the subcritical case.

\begin{proof}[Proof of Theorem~\ref{EXIST:GRS}] The proof is based on several steps, suitably adapted from the concentration-compactness method.
Rather than adapting deep,
but somewhat less elementary,
concentration-compactness principles of fractional flavor (such as those in~\cite{MR3216834, MR3617721, MR5074731}) to the case of the half-space, we provide here specific, self-contained arguments which can be entirely followed without any prior knowledge on the topic, and which can be easily extrapolated and used in other situations of interest.\medskip

\noindent{\bf Step 1. The minimizing sequence and auxiliary normalizations.}
We take a minimizing sequence~$u_k$. Since replacing $u_k$ by $|u_k|$ preserves the $L^{2^*_s}$-norm and does not increase the interaction energy, we may assume $u_k\ge0$. We also assume, without loss of generality, that it is normalized so that~$\|u_k\|_{L^{2^*_s}(\R^n_+)}=1$.
In particular,
\begin{equation}\label{SUMINSXD}
{\mathfrak{S}}_s=\lim_{k\to+\infty}
\frac{c_{n,s}}{2}\iint_{Q(\R^{n}_+)}\frac{(u_k(x)-u_k(y))^2}{|x-y|^{n+2s}}\,dx\,dy.
\end{equation}

There are a few additional normalizations that can be performed on any minimizing sequence.
First, given~$x_n \in(0,+\infty)$, one
can consider the Schwarz symmetric decreasing rearrangement of~$u_k(\cdot, x_n )$.
It is classical that this rearrangement perfectly preserves the norm in any Lebesgue space (thus, maintaining the normalization in~$L^{2^*_s}(\R^n_+)$). Moreover, see~\cite{MARIA}, this rearrangement decreases the term
$$ \iint_{Q(\R^{n}_+)}\frac{(u_k(x)-u_k(y))^2}{|x-y|^{n+2s}}\,dx\,dy$$
and therefore, without loss of generality, we can assume\footnote{As a side remark,
we observe that this normalization entails that
the minimizer constructed in Theorem~\ref{EXIST:GRS}, as well as the solution given in Corollary~\ref{0sqpjdlwbmlgtmlno2CO},
is rotationally symmetric and radially decreasing in the first~$(n-1)$ variables.}
that
\begin{equation}\label{RADSY} u_k(x',x_n)=u^\star_k(|x'|,x_n),\end{equation}
with~$\rho\longmapsto u^\star_k(\rho,x_n)$ nonincreasing in~$\rho$.

Moreover, since the nonlocal Neumann condition decreases the interaction energy (see~\cite[Theorem~2.1]{MR4651677},
without loss of generality we can suppose that, for all~$y\in\R^{n}_-$,
\begin{equation}\label{NSFQIFCELMI}
u_k(y)=\frac{\displaystyle\int_{\R^n_+}\frac{u_k(z)}{|y-z|^{n+2s}}\,dz}{\displaystyle\int_{\R^n_+}\frac{dz}{|y-z|^{n+2s}}}.
\end{equation}
As a consequence, using~\eqref{0pkjdmwr034iyhk92IKStuMD.3} with~$p:=2^*_s$, we have that
\begin{equation}\label{global-Lp-uk}
\sup_{k\in\N}\|u_k\|_{L^{2^*_s}(\R^n)}<+\infty.
\end{equation}
\medskip

\noindent{\bf Step 2. Mass split and nonvanishing at the origin.}
We pick~$\vartheta\in(0,1)$, to be chosen conveniently \label{JAOS:PAHEg}
small in what follows (in dependence of the energy gap provided by
Theorem~\ref{INterM0}).
We claim that the sequence~$u_k$ can be modified to avoid ``vanishing’’ as~$k\to+\infty$, namely that, without loss of generality, for all~$k\in\N$,
\begin{equation}\label{NOVAN}
\int_{B_1 \cap \R^n_+} |u_k(x)|^{2^*_s} \, dx =\vartheta.\end{equation}
Indeed, since
$$ \lim_{r\searrow0}\int_{B_r \cap \R^n_+} |u_k(x)|^{2^*_s} \, dx =0\quad{\mbox{and}}\quad\lim_{r\to+\infty}\int_{B_r \cap \R^n_+} |u_k(x)|^{2^*_s} \, dx =\|u_k\|_{L^{2^*_s}(\R^n_+)}^{2^*_s}=1,$$
for each~$k\in\N$ we can find~$r_k>0$ such that
\begin{equation}\label{OJASM:6v82d0c0k3SIKJ} \int_{B_{r_k} \cap \R^n_+} |u_k(x)|^{2^*_s} \, dx =\vartheta.\end{equation}
Hence, setting~$u_{k,r_k}(x):=r_k^{\frac{n}{2^*_s}} u_k(r_kx)$ and
using Lemma~\ref{SCAL}, with~$u:=u_k$, $\lambda=r_k$,
and~$\varrho=1$, we conclude that
$$
\int_{B_1\cap \R^n_+}|u_{k,r_k}(x)|^{2^*_s}\,dx=
\int_{B_{r_k}\cap \R^n_+}|u_k(x)|^{2^*_s}\,dx=\vartheta.$$
We also know from Lemma~\ref{SCAL} that~$u_{k,r_k}$
is a minimizing sequence as well, therefore, up to replacing~$u_k$ with~$u_{k,r_k}$, we can suppose that~\eqref{NSFQIFCELMI} and~\eqref{NOVAN} hold true.

As a consequence of~\eqref{NOVAN} we also have that
\begin{equation}\label{NOVANbis}
\int_{\R^n_+\setminus B_1} |u_k(x)|^{2^*_s} \, dx =1-\vartheta.\end{equation}
\medskip

\noindent{\bf Step 3. Continuity radii for the mass of the minimizing sequence.}
Due to the criticality of the problem, the mass of a minimizing sequence
could, in principle, concentrate along spheres. Some care is needed to avoid
this phenomenon. We claim that 
\begin{equation}\label{327sdf}
\liminf_{\rho\searrow1}\lim_{\e_0\searrow0}\lim_{k\to+\infty}\|u_{k}\|_{L^{2^*_s}(\R^n_+\cap(B_{\rho+\e_0}\setminus B_\rho))}=0.
\end{equation}
For this, let
$$ \R\ni r\longmapsto m_k(r):=
\begin{cases}
\|u_{k}\|^{2^*_s}_{L^{2^*_s}(\R^n_+\cap(B_{r}\setminus B_1))}&{\mbox{ if }}r>1\\0&{\mbox{ if }}r\le1.\end{cases}$$
We notice that~$m_k$ is nondecreasing and, by~\eqref{NOVAN},
$$m_k(r)\le\|u_{k}\|_{L^{2^*_s}(\R^n_+\setminus B_1)}^{2^*_s}=1-\vartheta.$$
Therefore, by Helly's Selection Theorem
(see e.g.~\cite[page~167]{MR385023}), there exists a nondecreasing function~$m_\infty$ such that,
for almost all $r\ge0$,
\begin{equation}\label{HELLDEF}\lim_{k\to+\infty} m_k(r) =m_\infty(r).\end{equation}

Since the set of discontinuity points of~$m_\infty$ is countable, we can take a sequence of continuity points~$r_j\searrow1$ as~$j\to+\infty$, namely
$$ m_\infty(r_j)=\lim_{\e_0\searrow0}m_\infty(r_j+\e_0).$$
That is,
\begin{eqnarray*}&&0=
\lim_{\e_0\searrow0}m_\infty(r_j+\e_0)-m_\infty(r_j)
=\lim_{\e_0\searrow0}\lim_{k\to+\infty}
m_k(r_j+\e_0)-m_k(r_j)\\&&\qquad\qquad
=\lim_{\e_0\searrow0}\lim_{k\to+\infty}
\|u_{k}\|^{2^*_s}_{L^{2^*_s}(\R^n_+\cap(B_{r_j+\e_0}\setminus B_1))}
-\|u_{k}\|^{2^*_s}_{L^{2^*_s}(\R^n_+\cap(B_{r_j}\setminus B_1))}\\&&\qquad\qquad
=\lim_{\e_0\searrow0}\lim_{k\to+\infty}
\|u_{k}\|^{2^*_s}_{L^{2^*_s}(\R^n_+\cap(B_{r_j+\e_0}\setminus B_{r_j}))}
.\end{eqnarray*}
Consequently,
$$ 0=\lim_{j\to+\infty}\lim_{\e_0\searrow0}\lim_{k\to+\infty}
\|u_{k}\|^{2^*_s}_{L^{2^*_s}(\R^n_+\cap(B_{r_j+\e_0}\setminus B_{r_j}))}
\ge\liminf_{\rho\searrow1}\lim_{\e_0\searrow0}\lim_{k\to+\infty}\|u_{k}\|_{L^{2^*_s}(\R^n_+\cap(B_{\rho+\e_0}\setminus B_\rho))}^{2^*_s}.$$
This proves~\eqref{327sdf}.\medskip

\noindent{\bf Step 4. The limit of the minimizing sequence.}
Being~$u_k$ a normalized minimizing sequence, we know that
\begin{equation}\label{3.7BIS}\sup_{k\in\N}
\iint_{Q(\R^{n}_+)} \frac{(u_k(x)-u_k(y))^2}{|x-y|^{n+2s}} \, dx\,dy<+\infty.\end{equation}
We denote by~$D^{s,2}(Q(\R^n_+))$ the energy space of measurable functions~$u$ such that~$u|_{\R^n_+}\in L^{2^*_s}(\R^n_+)$ and
$$\iint_{Q(\R^n_+)}\frac{(u(x)-u(y))^2}{|x-y|^{n+2s}}\,dx\,dy<+\infty,$$
endowed with the norm induced by the above quadratic form. Lemma~\ref{SOBLEE} shows that this is indeed a norm on the class under consideration.

%%%%%%More precisely, to make the weak compactness argument completely explicit, we understand~$D^{s,2}(Q(\R^n_+))$ as the Hilbert completion of the finite-energy class with respect to this norm. Lemma~\ref{SOBLEE} shows that the restriction map into~$L^{2^*_s}(\R^n_+)$ is continuous.

Up to a subsequence,~$u_k$ converges to some~$u_\infty$ weakly in~$D^{s,2}(Q(\R^n_+))$ and weakly in~$L^{2^*_s}(\R^n_+)$,
strongly in~$L^{q}(\R^n_+\cap B_\varrho)$ for every~$\varrho>0$ and~$q\in[1,2^*_s)$, and almost everywhere in~$\R^n_+$ (see e.g.~\cite[Theorem~7.1]{MR2944369} for the relevant local compact embedding).

Moreover, for every fixed~$y\in\R^n_-$, the function~$z\mapsto |y-z|^{-n-2s}$ belongs to
the dual of~$L^{2^*_s}(\R^n_+)$. Hence, the weak~$L^{2^*_s}$ convergence and~\eqref{NSFQIFCELMI} imply that~$u_k(y)\to u_\infty(y)$, where~$u_\infty$ is defined in~$\R^n_-$ by~\eqref{NSFQIFCELMI}.
Thus, we have that~$u_k\to u_\infty$ almost everywhere in~$\R^n$.

Consequently, by~\eqref{SUMINSXD} and Fatou's Lemma,
\begin{equation}\label{ILPA:1}
{\mathfrak{S}}_s\ge\frac{c_{n,s}}{2} \iint_{Q(\R^n_+) } \frac{(u_\infty(x)-u_\infty(y))^2}{|x-y|^{n+2s}} \, dx \, dy .
\end{equation}\medskip

\noindent{\bf Step 5. Energy split for the vanishing case.}
Our objective will be to show that~$u_\infty$ does not vanish identically.
This is a delicate point and requires a careful energy analysis. For this,
we point out that when~$u_\infty$ vanishes, then the energy of
a minimizing sequence can be efficiently separated into localized components with
``almost no interaction''.
The precise statement goes as follows. Let~$\phi\in C^\infty_c(B_2,[0,1])$,
and define~$u_{1,k,\phi}:=u_k\phi$ and~$u_{2,k,\phi}:=u_k(1-\phi)$. In this setting,
\begin{equation}\label{splTYH:1}\begin{split}&
{\mbox{if $u_\infty$ vanishes identically, then }}\\&
\frac{2{\mathfrak{S}}_s}{c_{n,s}}=
\lim_{k\to+\infty}\left[\sum_{j=1}^2
\iint_{Q(\R^n_+) } \frac{(u_{j,k,\phi}(x)-u_{j,k,\phi}(y))^2}{|x-y|^{n+2s}} \, dx \, dy+2J_{1,k,\phi}\right],
\end{split}\end{equation}
where
\begin{equation}\label{J1ojlwds}
J_{1,k,\phi}:=\iint_{{Q(\R^n_+)} }\frac{
\phi(x)(1-\phi(x))(u_k(x)-u_k(y))^2
}{|x-y|^{n+2s}}\,dx\,dy\ge0.\end{equation}

To prove this claim, first we check that
\begin{equation}\label{DOPPIp}
\lim_{k\to+\infty}\left|
\iint_{Q(\R^n_+) } \frac{(u_{1,k,\phi}(x)-u_{1,k,\phi}(y))(u_{2,k,\phi}(x)-u_{2,k,\phi}(y))}{|x-y|^{n+2s}} \, dx \, dy
-J_{1,k,\phi}\right|=0.
\end{equation}
To accomplish this goal, 
one can expand the increments as follows: 
\begin{align*}
u_{1,k,\phi}(x)-u_{1,k,\phi}(y)
&=\phi(x)(u_k(x)-u_k(y))
      +(\phi(x)-\phi(y))u_k(y),\\
u_{2,k,\phi}(x)-u_{2,k,\phi}(y)
&=(1-\phi(x))(u_k(x)-u_k(y))
      -(\phi(x)-\phi(y))u_k(y).
\end{align*} 
In this way,
\begin{equation}\label{2dLAOJMul}
\left|\iint_{Q(\R^n_+) }\frac{(u_{1,k,\phi}(x)-u_{1,k,\phi}(y))(u_{2,k,\phi}(x)-u_{2,k,\phi}(y))}{|x-y|^{n+2s}} \, dx \, dy-J_{1,k,\phi}\right|\le\sum_{m=2}^4 J_{m,k,\phi}
,\end{equation}
where
\begin{eqnarray*}
&& J_{2,k,\phi}:=\iint_{Q(\R^n_+) }\frac{
\phi(x)\,|\phi(x)-\phi(y)|\, |u_k(y)|\,|u_k(x)-u_k(y)|}{|x-y|^{n+2s}}\,dx\,dy,\\
&& J_{3,k,\phi}:=\iint_{{Q(\R^n_+)} }\frac{
(1-\phi(x))\,|\phi(x)-\phi(y)|\,|u_k(y)|\,
|u_k(x)-u_k(y)|
}{|x-y|^{n+2s}}\,dx\,dy,\\
{\mbox{and }}&& J_{4,k,\phi}:=\iint_{Q(\R^n_+) }\frac{
(\phi(x)-\phi(y))^2
|u_k(y)|^2}{|x-y|^{n+2s}}\,dx\,dy.
\end{eqnarray*}

Now, using the H\"older Inequality and recalling~\eqref{3.7BIS}, we see that
\begin{equation}\label{OIJK:0qudj-97bv6nd5h939356tfg10dj}\begin{split}
J_{2,k,\phi}+J_{3,k,\phi}&=\iint_{{Q(\R^n_+)} }\frac{|\phi(x)-\phi(y)|\,|u_k(y)|\,
|u_k(x)-u_k(y)|
}{|x-y|^{n+2s}}\,dx\,dy\\&
\le\sqrt{ J_{4,k,\phi}\;
\iint_{{Q(\R^n_+)} }\frac{(u_k(x)-u_k(y))^2}{|x-y|^{n+2s}}\,dx\,dy
}\\& \le C\,\sqrt{ J_{4,k,\phi}}.
\end{split}\end{equation}

Moreover, given~$R>2$, we have that
\begin{equation}\label{fDSUOJD2UDsUJ}\begin{split}&
\lim_{k\to+\infty}\iint_{\R^n\times B_R }\frac{
(\phi(x)-\phi(y))^2
|u_k(y)|^2}{|x-y|^{n+2s}}\,dx\,dy\\&\qquad\le 
C\lim_{k\to+\infty}\iint_{\R^n\times B_R }\frac{
\min\{1,|x-y|^2\}
|u_k(y)|^2}{|x-y|^{n+2s}}\,dx\,dy\\&\qquad\le
C\lim_{k\to+\infty}\int_{ B_R }
|u_k(y)|^2\,dy\\&\qquad=C\int_{ B_R }
|u_\infty(y)|^2\,dy\\&\qquad
=0,\end{split}
\end{equation}
due to the local convergence for subcritical exponents (recall also the uniform bound~\eqref{global-Lp-uk}) and the vanishing assumption on~$u_\infty$.

Furthermore, using the H\"older Inequality with exponents~$\frac{2^*_s}{2}=\frac{n}{n-2s}$ and~$\frac{n}{2s}$,
\begin{equation*}
\begin{split}&
\iint_{\R^n\times(\R^n\setminus B_R) }\frac{
(\phi(x)-\phi(y))^2
|u_k(y)|^2}{|x-y|^{n+2s}}\,dx\,dy
=\iint_{B_2\times(\R^n\setminus B_R) }\frac{
\phi^2(x)
|u_k(y)|^2}{|x-y|^{n+2s}}\,dx\,dy\\&\quad\le C
\int_{\R^n\setminus B_R }\frac{
|u_k(y)|^2}{|y|^{n+2s}}\,dy\le\|u_k\|^2_{L^{2^*_s}(\R^n\setminus B_R)}
\left(\int_{\R^n\setminus B_R}\frac{dy}{|y|^{\frac{n(n+2s)}{2s}}}\right)^{\frac{2s}n}\le \frac{C}{R^n},
\end{split}\end{equation*}
where the last inequality follows from the uniform global bound~\eqref{global-Lp-uk}.

This and~\eqref{fDSUOJD2UDsUJ} give that
$$\lim_{k\to+\infty}J_{4,k,\phi}\le\frac{C}{R^n}.$$
Since we can now send~$R\to+\infty$, we thereby conclude that
$$\lim_{k\to+\infty}J_{4,k,\phi}=0$$
and accordingly, in light of~\eqref{OIJK:0qudj-97bv6nd5h939356tfg10dj},
\begin{equation*}
\lim_{k\to+\infty}
J_{2,k,\phi}+J_{3,k,\phi}+J_{4,k,\phi}=0.\end{equation*}
The desired claim in~\eqref{DOPPIp} now follows from~\eqref{2dLAOJMul}.

Now we point out that~$u_k=u_{1,k,\phi}+u_{2,k,\phi}$, thus
\begin{eqnarray*}&&
\iint_{Q(\R^n_+) } \frac{(u_k(x)-u_k(y))^2}{|x-y|^{n+2s}} \, dx \, dy\\&&\quad=\sum_{j=1}^2
\iint_{Q(\R^n_+) } \frac{(u_{j,k,\phi}(x)-u_{j,k,\phi}(y))^2}{|x-y|^{n+2s}} \, dx \, dy\\&&\qquad\qquad
+2
\iint_{Q(\R^n_+) } \frac{(u_{1,k,\phi}(x)-u_{1,k,\phi}(y))(u_{2,k,\phi}(x)-u_{2,k,\phi}(y))}{|x-y|^{n+2s}} \, dx \, dy
.\end{eqnarray*}
From this, \eqref{SUMINSXD},
and~\eqref{DOPPIp}, we obtain the desired result in~\eqref{splTYH:1}.
\medskip

\noindent{\bf Step 6. Avoiding vanishing in the limit.}
We claim that
\begin{equation}\label{NINBZA}
{\mbox{$u_\infty$ does not vanish identically in~$\R^n_+$.}}
\end{equation}
Indeed, suppose, by contradiction that
\begin{equation}\label{NINBZA:SCOhy}\begin{split}&
{\mbox{$u_\infty$ vanishes identically in~$\R^n_+$,}}
\\&{\mbox{and thus in~$\R^n$, by the normalization in~\eqref{NSFQIFCELMI}.}}\end{split}
\end{equation}

Let~$\rho\in(1,\frac32)$ be a continuity point for the function~$m_\infty$ introduced in~\eqref{HELLDEF}. Choose~$\e\in(0,2-\rho)$ such that~$\rho+\e$ is also a continuity point for~$m_\infty$.
Let~$\eta_{\e,\rho}\in C^\infty_c(B_{\rho+\e},[0,1])$ with~$\eta_{\e,\rho}=1$ in~$B_\rho$.
We define~$u_{1,k,\e,\rho}:=u_k \eta_{\e,\rho}$
and~$u_{2,k,\e,\rho}:=u_k(1- \eta_{\e,\rho})$.

Accordingly, by~\eqref{NINBZA:SCOhy}, we can use the energy splitting in~\eqref{splTYH:1}.
This, together with~\eqref{J1ojlwds}, entails that
\begin{equation*}\begin{split}\frac{2{\mathfrak{S}}_s}{c_{n,s}}&\ge
\liminf_{k\to+\infty}\sum_{j=1}^2
\iint_{Q(\R^n_+) } \frac{(u_{j,k,\e,\rho}(x)-u_{j,k,\e,\rho}(y))^2}{|x-y|^{n+2s}} \, dx \, dy
\\&\ge\frac{2{\mathfrak{S}}_s}{c_{n,s}}\liminf_{k\to+\infty}\sum_{j=1}^2
\|u_{j,k,\e,\rho}\|^2_{L^{2^*_s}(\R^n_+)}
\\&\ge\frac{2{\mathfrak{S}}_s}{c_{n,s}}\lim_{k\to+\infty}\Big(
\|u_{k}\|^2_{L^{2^*_s}(B_\rho\cap\R^n_+)}+
\|u_{k}\|^2_{L^{2^*_s}(\R^n_+\setminus B_{\rho+\e})}\Big)\\&=\frac{2{\mathfrak{S}}_s}{c_{n,s}}\lim_{k\to+\infty}\left(
\|u_{k}\|^2_{L^{2^*_s}(B_\rho\cap\R^n_+)}+
\left( \|u_{k}\|^{2^*_s}_{L^{2^*_s}(\R^n_+\setminus B_{\rho})}
-\|u_{k}\|^{2^*_s}_{L^{2^*_s}(\R^n_+\cap(B_{\rho+\e}\setminus B_{\rho}))}
\right)^{\frac2{2^*_s}}\right)
.\end{split}\end{equation*}

We now let~$\e\searrow0$ through values such that~$\rho+\e$ is a continuity point of~$m_\infty$. By the continuity of~$m_\infty$ at~$\rho$, the mass of the annulus~$B_{\rho+\e}\setminus B_\rho$ tends to zero. We then let~$\rho\searrow1$ through continuity points of~$m_\infty$ (up to a subsequence), and we conclude that
\begin{equation*}\begin{split}\frac{2{\mathfrak{S}}_s}{c_{n,s}}&\ge
\frac{2{\mathfrak{S}}_s}{c_{n,s}}\left(
a^{\frac2{2^*_s}}+
b^{\frac2{2^*_s}}\right)
,\end{split}\end{equation*}
with
$$ a:=\lim_{\rho\searrow1}\lim_{k\to+\infty}\|u_{k}\|^{2^*_s}_{L^{2^*_s}(B_\rho\cap\R^n_+)}\qquad{\mbox{and}}\qquad
b:=\lim_{\rho\searrow1}\lim_{k\to+\infty}\|u_{k}\|^{2^*_s}_{L^{2^*_s}(\R^n_+\setminus B_{\rho})}.$$

Hence, in view of Lemma~\ref{SOBLEE},
\begin{equation}\label{o2jdwlmef:0wedf}
1\ge a^{\frac2{2^*_s}}+
b^{\frac2{2^*_s}}.\end{equation}
We also remark that
\begin{equation}\label{AOJSb} a+b=\lim_{\rho\searrow1}\lim_{k\to+\infty}\|u_{k}\|^{2^*_s}_{L^{2^*_s}(B_\rho\cap\R^n_+)}
+\|u_{k}\|^{2^*_s}_{L^{2^*_s}(\R^n_+\setminus B_{\rho})}
=\lim_{\rho\searrow1}
\lim_{k\to+\infty}\|u_{k}\|^{2^*_s}_{L^{2^*_s}(\R^n_+)}=1,\end{equation}
thanks to the normalization of our minimizing sequence,
therefore~\eqref{o2jdwlmef:0wedf} can be written as
\begin{equation*}
\alpha^{\frac{2^*_s}2} +\beta^{\frac{2^*_s}2}=
a+b=1^{\frac{2^*_s}2}\ge\left( a^{\frac2{2^*_s}}+
b^{\frac2{2^*_s}}\right)^{\frac{2^*_s}2}=( \alpha+\beta)^{\frac{2^*_s}2},\end{equation*}
where~$\alpha:=a^{\frac2{2^*_s}}$ and~$\beta:=
b^{\frac2{2^*_s}}$.

Hence, by Lemma~\ref{LE:CONVEX}, either~$\alpha=0$ or~$\beta=0$, and thus
either~$a=0$ or~$b=0$.

However, by~\eqref{NOVAN},
$$ a\ge\lim_{k\to+\infty}\|u_{k}\|^{2^*_s}_{L^{2^*_s}(B_1\cap\R^n_+)}=\vartheta>0,$$
therefore necessarily
\begin{equation}\label{OLJSs6ai:Be}
0=b=\lim_{\rho\searrow1}\lim_{k\to+\infty}\|u_{k}\|^{2^*_s}_{L^{2^*_s}(\R^n_+\setminus B_{\rho})}=1-\vartheta-
\lim_{\rho\searrow1}\lim_{k\to+\infty}\|u_{k}\|^{2^*_s}_{L^{2^*_s}(\R^n_+\cap(B_\rho\setminus B_{1}))}
,\end{equation}
where we have also used~\eqref{NOVANbis}.

Now, let~$\tau\in[0,1]$ and~$x_n\in[0,1-\sqrt\tau]$. We see that
\begin{equation} \label{HSJKD:03-1}\sqrt\tau\le 1-x_n\le1-x_n^2.\end{equation}
This implies that
$$(1+\tau)^2-x_n^2=1+2\tau+\tau^2-x_n^2\le 
1-x_n^2+2(1-x_n^2)^2+(1-x_n^2)^4\le 4(1-x_n^2)$$
and therefore
\begin{equation} \label{HSJKD:03-2}
\sqrt{(1+\tau)^2-x_n^2}\le 2\sqrt{1-x_n^2}.
\end{equation}

Moreover,
\begin{equation}\label{cnbvueworyt84ygtkjha65748}\begin{split}&
\sqrt{(1+\tau)^2-x_n^2}-\sqrt{1-x_n^2}=
\frac{\big((1+\tau)^2-x_n^2\big)-\big(1-x_n^2\big)}{\sqrt{(1+\tau)^2-x_n^2}+\sqrt{1-x_n^2}}
\\&\qquad=\frac{2\tau+\tau^2}{\sqrt{(1+\tau)^2-x_n^2}+\sqrt{1-x_n^2}}\le
\frac{3\tau}{ \sqrt{1-x_n^2}}.
\end{split}
\end{equation}

Thus, we deduce from~\eqref{RADSY},  \eqref{HSJKD:03-1},\eqref{HSJKD:03-2} and~\eqref{cnbvueworyt84ygtkjha65748}
that, if~$n\ge2$, $\tau\in[0,1]$, and~$x_n\in[0,1-\sqrt\tau]$,
\begin{equation}\label{POKSOJP15mcwe-1tuy}\begin{split}&
\int_{\{|x'|\in[\sqrt{1-x_n^2},\sqrt{(1+\tau)^2-x_n^2}]\}} |u_k(x',x_n)|^{2^*_s}\,dx'
\le C \int_{\sqrt{1-x_n^2}}^{\sqrt{(1+\tau)^2-x_n^2}} r^{n-2}\,|u^\star_k(r,x_n)|^{2^*_s}\,dr\\&\qquad\qquad
\le C (\sqrt{(1+\tau)^2-x_n^2})^{n-2}\int_{\sqrt{1-x_n^2}}^{\sqrt{(1+\tau)^2-x_n^2}} |u^\star_k(r,x_n)|^{2^*_s}\,dr\\&\qquad\qquad
\le C (\sqrt{1-x_n^2})^{n-2}\int_{\sqrt{1-x_n^2}}^{\sqrt{(1+\tau)^2-x_n^2}} |u^\star_k(r,x_n)|^{2^*_s}\,dr
\\&\qquad\qquad
\le C (\sqrt{1-x_n^2})^{n-2}\int_{\sqrt{1-x_n^2}}^{\sqrt{(1+\tau)^2-x_n^2} } |u^\star_k(\sqrt{1-x_n^2},x_n)|^{2^*_s}\,dr\\&\qquad\qquad\le C \tau(\sqrt{1-x_n^2})^{n-3} |u^\star_k(\sqrt{1-x_n^2},x_n)|^{2^*_s}\\&\qquad\qquad= C \tau (\sqrt{1-x_n^2})^{n-4} \int_{\frac{\sqrt{1-x_n^2}}2}^{\sqrt{1-x_n^2}}
|u^\star_k(\sqrt{1-x_n^2},x_n)|^{2^*_s}\,dr\\&\qquad\qquad\le
C \tau(\sqrt{1-x_n^2})^{n-4}\int_{\frac{\sqrt{1-x_n^2}}2}^{\sqrt{1-x_n^2}}|
u^\star_k(r,x_n)|^{2^*_s}\,dr\\&\qquad\qquad\le\frac{ C \tau}{1-x_n^2} \int_{\frac{\sqrt{1-x_n^2}}2}^{\sqrt{1-x_n^2}}
r^{n-2} |u^\star_k(r,x_n)|^{2^*_s}\,dr\\&\qquad\qquad\le
\frac{C \tau}{1-x_n^2} \int_{\{|x'|<\sqrt{1-x_n^2}\}}| u_k(x',x_n)|^{2^*_s}\,dx'\\
\\&\qquad\qquad\le
C \sqrt\tau \int_{\{|x'|<\sqrt{1-x_n^2}\}}| u_k(x',x_n)|^{2^*_s}\,dx'
.\end{split}
\end{equation}

We also observe that if~$x=(x',x_n)\in B_{1+\tau}\setminus B_1$, then
$$ |x'|\in\left[ \sqrt{1-x_n^2}, \sqrt{(1+\tau)^2-x_n^2} \right] .$$
As a result, we deduce from~\eqref{POKSOJP15mcwe-1tuy} that,
for all~$\tau\in[0,1]$, if~$n\ge2$,
\begin{equation}\label{KS:D03-4kKSMMD-305yg4Nsz}
\begin{split}&\int_{(B_{1+\tau}\setminus B_1)\cap\{x_n\in[0,1-\sqrt\tau]\}}|u_k(x)|^{2^*_s}\,dx\\&\quad\le
\int_{\{|x'|\in[\sqrt{1-x_n^2},\sqrt{(1+\tau)^2-x_n^2}]\}\times\{{x_n\in [0,1-\sqrt\tau]}\}} |u_k(x',x_n)|^{2^*_s}\,dx'\,dx_n\\&\quad
\le C\sqrt\tau \int_{\{|x'|<\sqrt{1-x_n^2}\}\times\{x_n\in[0,1-\sqrt\tau]\}}  |u_k(x',x_n)|^{2^*_s}\,dx'\,dx_n\\&\quad\le C\sqrt\tau,
\end{split}
\end{equation}
and a similar result holds when~$n=1$ because~$(B_{1+\tau}\setminus B_1)\cap\{x_n\in[0,1-\sqrt\tau]\}=\varnothing$.

Furthermore, if~$x\in(B_{1+\tau}\setminus B_{1})\cap\{x_n>1-\sqrt\tau\}$, then
\begin{eqnarray*}&&
|x-e_n|^2=|x'|^2+(1-x_n)^2=|x|^2-x_n^2+(1-x_n)^2\\&&\quad
\le(1+\tau)^2-(1-\sqrt\tau)^2+\tau=
2\sqrt\tau+2\tau+\tau^2<16\sqrt\tau,\end{eqnarray*}
as long as~$\tau$ is small enough, leading to
$$(B_{1+\tau}\setminus B_{1})\cap\{x_n>1-\sqrt\tau\}\subseteq B_{4\sqrt[4]{\tau}}(e_n).$$
Consequently, by~\eqref{KS:D03-4kKSMMD-305yg4Nsz},
\begin{eqnarray*}&&\lim_{\rho\searrow1}
\lim_{k\to+\infty}\|u_{k}\|^{2^*_s}_{L^{2^*_s}(\R^n_+\cap(B_\rho\setminus B_{1}))}\\&&\quad=
\lim_{\tau\searrow0}
\lim_{k\to+\infty}\left[\|u_{k}\|^{2^*_s}_{L^{2^*_s}((B_{1+\tau}\setminus B_{1})\cap\{x_n[0,1-\sqrt\tau]\})}+\|u_{k}\|^{2^*_s}_{L^{2^*_s}((B_{1+\tau}\setminus B_{1})\cap\{x_n>1-\sqrt\tau\})}\right]
\\&&\quad=\lim_{\tau\searrow0}
\lim_{k\to+\infty}\|u_{k}\|^{2^*_s}_{L^{2^*_s}((B_{1+\tau}\setminus B_{1})\cap\{x_n>1-\sqrt\tau\})}\\&&\quad\le\lim_{\tau\searrow0}
\lim_{k\to+\infty}\|u_{k}\|^{2^*_s}_{L^{2^*_s}(B_{4\sqrt[4]{\tau}}(e_n))}.
\end{eqnarray*}

{F}rom this and~\eqref{OLJSs6ai:Be} we arrive at
\begin{equation*}
\lim_{\tau\searrow0}
\lim_{k\to+\infty}\|u_{k}\|^{2^*_s}_{L^{2^*_s}(B_{4\sqrt[4]{\tau}}(e_n))}\ge1-\vartheta.
\end{equation*}
We thus\footnote{As anticipated on page~\pageref{JAOS:PAHEg},
the mass parameter~$\vartheta$
will be chosen conveniently small: here we are assuming, to start with,
that~$\vartheta<\frac12$.}
select~$\tau_0\in\left(0,\frac1{1000}\right)$ sufficiently small such that,
for all~$\tau\in(0,\tau_0]$ and~$k$ sufficiently large,
\begin{equation}\label{cjsaf74865t7iugfwsaigf8765jhf}
\|u_{k}\|^{2^*_s}_{L^{2^*_s}(B_{4\sqrt[4]{\tau}}(e_n))}\ge1-2\vartheta.
\end{equation}

We now consider a cutoff function~$\varphi_\tau\in C^\infty_c(B_{5\sqrt[4]{\tau}}(e_n),[0,1])$
with~$\varphi_\tau:=1$ in~$B_{4\sqrt[4]{\tau}}(e_n)$. We define~$v_{k,1}:=
u_k\varphi_\tau$ and~$v_{k,2}:=u_k(1-\varphi_\tau)$.

We stress that~$v_{k,1}(x)=0$ for all~$x\in\R^n_-$. 
As a consequence, by the energy split in~\eqref{splTYH:1}
and~\eqref{J1ojlwds},
\begin{align*}
{\mathfrak{S}}_s
&\ge\lim_{k\to+\infty}\sum_{j=1}^2
\frac{c_{n,s}}{2}\iint_{Q(\R^{n}_+)}\frac{(v_{k,j}(x)-v_{k,j}(y))^2}{|x-y|^{n+2s}}\,dx\,dy\\
&\ge\lim_{k\to+\infty}
\frac{c_{n,s}}{2}\iint_{Q(\R^{n}_+)}\frac{(v_{k,1}(x)-v_{k,1}(y))^2}{|x-y|^{n+2s}}\,dx\,dy\\
&=\lim_{k\to+\infty}\frac{c_{n,s}}{2}\iint_{\R^{2n}}\frac{(v_{k,1}(x)-v_{k,1}(y))^2}{|x-y|^{n+2s}}\,dx\,dy\\
&\ge {\mathfrak{S}}_s^\star\lim_{k\to+\infty}\|v_{k,1}\|_{L^{2^*_s}(\R^n)}^2\\
&\ge {\mathfrak{S}}_s^\star\lim_{k\to+\infty}\|u_{k}\|_{L^{2^*_s}(B_{4\sqrt[4]{\tau}}(e_n))}^2\\
&\ge {\mathfrak{S}}_s^\star(1-2\vartheta)^{\frac2{2^*_s}},
\end{align*}
where in the last inequality we have used~\eqref{cjsaf74865t7iugfwsaigf8765jhf}.

A contradiction is now reached by choosing~$\vartheta>0$ sufficiently small such that
$$ {\mathfrak{S}}_s^\star(1-2\vartheta)^{\frac2{2^*_s}}\ge
{\mathfrak{S}}_s+\vartheta,$$
and we stress that this is possible thanks to Theorem~\ref{INterM0}.
This proves~\eqref{NINBZA}.

Now we claim that
\begin{equation}\label{ILPA:2}
\|u_\infty\|_{L^{2^*_s}(\R^n_+)}=1.
\end{equation}
Indeed, let $$ \sigma:=\|u_\infty\|_{L^{2^*_s}(\R^n_+)}\quad{\mbox{ and }}\quad\tau:=\lim_{k\to+\infty}\|u_\infty-u_k\|_{L^{2^*_s}(\R^n_+)}.$$ 
By the weak convergence of~$u_k$ in~$D^{s,2}(Q(\R^{n}_+))$, we know that
\begin{eqnarray*}0&=&
c_{n,s}\lim_{k\to+\infty}
\iint_{Q(\R^n_+) } \frac{\big(u_\infty(x)-u_\infty(y)\big)
\big((u_\infty-u_k)(x)-(u_\infty-u_k)(y)\big)
}{|x-y|^{n+2s}} \, dx \, dy
\\&=&
\lim_{k\to+\infty}
\frac{c_{n,s}}{2} \iint_{Q(\R^n_+) } \frac{(u_\infty(x)-u_\infty(y))^2}{|x-y|^{n+2s}} \, dx \, dy\\&&\quad+
\frac{c_{n,s}}{2} \iint_{Q(\R^n_+) } \frac{\big((u_\infty-u_k)(x)-(u_\infty-u_k)(y)\big)^2}{|x-y|^{n+2s}} \, dx \, dy\\&&\quad-
\frac{c_{n,s}}{2} \iint_{Q(\R^n_+) } \frac{(u_k(x)-u_k(y))^2}{|x-y|^{n+2s}} \, dx \, dy\\&\ge& \lim_{k\to+\infty} {\mathfrak{S}}_s\sigma^2+
{\mathfrak{S}}_s\|u_\infty-u_k\|_{L^{2^*_s}(\R^n_+)}^2
-
\frac{c_{n,s}}{2} \iint_{Q(\R^n_+) } \frac{(u_k(x)-u_k(y))^2}{|x-y|^{n+2s}} \, dx \, dy\\&=&
{\mathfrak{S}}_s\sigma^2+{\mathfrak{S}}_s\tau^2-{\mathfrak{S}}_s.
\end{eqnarray*}
Therefore,  recalling Lemma~\ref{SOBLEE}, by \eqref{SUMINSXD}
$$ 1\ge\sigma^2+\tau^2.$$

In addition, by the Brezis-Lieb Lemma, 
\begin{equation}\label{khn:swd2BA}\begin{split}&\sigma^{2^*_s}=
\int _{\R^n_+}|u_\infty(x)|^{2^*_s}\,dx =\lim _{k\to+ \infty } \int _{\R^n_+}|u_k(x)|^{2^*_s}\,dx -\int _{\R^n_+}|u_\infty(x)-u_k(x)|^{2^*_s}\,dx
\\&\qquad\qquad\qquad=1-\lim _{k\to+ \infty }\int _{\R^n_+}|u_\infty(x)-u_k(x)|^{2^*_s}\,dx=1-\tau^{2^*_s}.
\end{split}\end{equation}

All in all, defining~$q:=\frac{2^*_s}2>1$,
we have that
\begin{equation}\label{FVA6e0ispq94} \big(\sigma^2+\tau^2\big)^{q}\le1=\sigma^{2^*_s}+\tau^{2^*_s}=
\sigma^{2q}+\tau^{2q}.\end{equation}
Now we exploit~\eqref{NINBZA} in tandem with Lemma~\ref{LE:CONVEX} (used here with~$\alpha:=\sigma^2>0$ and~$\beta:=\tau^2$).
In this way, we deduce from~\eqref{FVA6e0ispq94} that~$
\tau=0$.
Accordingly, recalling~\eqref{khn:swd2BA}, we find that~$\sigma=1$.
This completes the proof of~\eqref{ILPA:2}.

In view of~\eqref{ILPA:1} and~\eqref{ILPA:2}, we see that~$u_\infty$ is the desired minimum.
\end{proof}

\section{Proof of Corollary~\ref{0sqpjdlwbmlgtmlno2CO}}\label{0sqpjdlwbmlgtmlno2COse}

The proof of Corollary~\ref{0sqpjdlwbmlgtmlno2CO} is easy, but we
provide full details for the facility of the reader.

\begin{proof}[Proof of Corollary~\ref{0sqpjdlwbmlgtmlno2CO}]
Let~$u$ be the minimizer in Theorem~\ref{EXIST:GRS}
such that~$\|u\|_{L^{2^*_s}(\R^n_+)}=1$. Since the absolute value decreases the Gagliardo seminorm,
without loss of generality we can suppose that~$u\ge0$.
Let~$\e\in\R$ and~$\phi\in C^\infty_c(\R^n)$.
Then, setting~$u_\e:=u+\e\phi\in L^{2^*_s}(\R^n_+)$,
\begin{equation*} \begin{split}&
{\mathfrak{S}}_s\,\left(
1+2\e
\int_{\R^n_+}u^{2^*_s-1}(x)\phi(x)\,dx+o(\e)\right)=
{\mathfrak{S}}_s\,\|u_\e\|_{L^{2^*_s}(\R^n_+)}^2 \\&\quad\le
\frac{c_{n,s}}{2}\iint_{Q(\R^{n}_+)}\frac{(u_\e(x)-u_\e(y))^2}{|x-y|^{n+2s}}\,dx\,dy\\&\quad
={\mathfrak{S}}_s+c_{n,s}\e\iint_{Q(\R^{n}_+)}\frac{(u(x)-u(y))(\phi(x)-\phi(y))}{|x-y|^{n+2s}}\,dx\,dy+o(\e).
\end{split}\end{equation*}
Consequently, we factorize $\e$ and pass to the limit, recalling Lemma~\ref{SOBLEE},
$$ \int_{\R^n_+}u^{2^*_s-1}(x)\phi(x)\,dx=\frac{c_{n,s}}{{2\mathfrak{S}}_s}
\iint_{Q(\R^{n}_+)}\frac{(u(x)-u(y))(\phi(x)-\phi(y))}{|x-y|^{n+2s}}\,dx\,dy,$$
yielding that, in the distributional sense,
\begin{equation*}
\begin{cases} (-\Delta)^s u=2{\mathfrak{S}}_s\,u^{2^*_s-1} &{\mbox{ in }}\R^n_+,\\
{\mathcal{N}}_s u=0&{\mbox{ in }}\R^n_-.\end{cases}
\end{equation*}
The function~$(2{\mathfrak{S}}_s)^{\frac{1}{2^*_s-2}}u$ is thus a nontrivial solution of~\eqref{pb1}.
\end{proof}

\begin{appendix}

\section{Auxiliary material of general interest}

\subsection{Simple observations about convexity}
This is a simple, but useful, remark of general use:

\begin{lemma}\label{LE:CONVEX0}
Let~$\alpha$, $\beta\ge0$ and~$q\ge1$. Then,
\begin{equation}\label{LE:CONVEX09} \big(\alpha+\beta\big)^{q}\ge
\alpha^{q}+\beta^{q}.\end{equation}
\end{lemma}

\begin{proof} Up to exchanging~$\alpha$ and~$\beta$, we can suppose that~$\alpha\ge\beta$. Then, using the convexity of the map~$\R_+\ni r\mapsto r^q$,
$$ \big(\alpha+\beta\big)^{q}\ge\alpha^q+q\alpha^{q-1}\beta\ge
\alpha^q+q\beta^q\ge\alpha^q+\beta^q,$$
as desired.
\end{proof}

The next result clarifies that when~$q>1$, equality is attained in~\eqref{LE:CONVEX09}
if and only if either~$\alpha=0$ or~$\beta=0$.

\begin{lemma}\label{LE:CONVEX}
Let~$\alpha$, $\beta\ge0$ and~$q>1$. Suppose that
\begin{equation}\label{OJLsn-1} \big(\alpha+\beta\big)^{q}\le
\alpha^{q}+\beta^{q}.\end{equation}
Then, either~$\alpha=0$ or~$\beta=0$.
\end{lemma}

\begin{proof} For the sake of contradiction, we suppose that
\begin{equation}\label{CONTRAJOLDMD}
{\mbox{both~$\alpha$ and~$\beta$ are strictly positive.}}\end{equation}
Using the strict convexity of the map~$\R_+\ni r\mapsto r^q$, we know that~$(1+a)^q>1+qa$, for all~$a>0$. We can use this both
with~$a:=\frac{\beta}{\alpha}$ and~$a:=\frac{\alpha}{\beta}$, which
is possible due to~\eqref{CONTRAJOLDMD}, and obtain that
$$ \big(\alpha+\beta\big)^{q}>\max\big\{
\alpha^{q}+q\beta\alpha^{q-1}, \beta^{q}+q\alpha\beta^{q-1}\big\}.$$
This, in tandem with~\eqref{OJLsn-1}, returns that
$$ \max\big\{
\alpha^{q}+q\beta\alpha^{q-1}, \beta^{q}+q\alpha\beta^{q-1}\big\}<\alpha^{q}+\beta^{q}.$$
Consequently,
$$ q\alpha^{q-1}<\beta^{q-1}\qquad{\mbox{and}}\qquad
q\beta^{q-1}<\alpha^{q-1},$$
leading to~$q^2\alpha^{q-1}<\alpha^{q-1}$, and thus~$q<1$, which
is a contradiction.
\end{proof}

\subsection{Scaling invariance}
We recall a useful rescaling argument:

\begin{lemma}\label{SCAL}
If~$\lambda>0$ and~$u_\lambda(x):=\lambda^{\frac{n}{2^*_s}} u(\lambda x)
=\lambda^{\frac{n-2s}{2}} u(\lambda x)$, we have that
$$ \iint_{Q(\R^{n}_+)}\frac{(u_\lambda(x)-u_\lambda(y))^2}{|x-y|^{n+2s}}\,dx\,dy=
\iint_{Q(\R^{n}_+)}\frac{(u(x)-u(y))^2}{|x-y|^{n+2s}}\,dx\,dy,$$
$$
\int_{B_\varrho\cap \R^n_+}|u_\lambda(x)|^{2^*_s}\,dx=
\int_{B_{\lambda\varrho}\cap \R^n_+}|u(x)|^{2^*_s}\,dx,
$$
and
$$\|u_\lambda\|_{L^{2^*_s}(\R^n_+)}=\|u\|_{L^{2^*_s}(\R^n_+)}.$$
\end{lemma}

The proof of  Lemma~\ref{SCAL} relies on a simple change of variable and is omitted.

\subsection{Cutoff arguments and kernel estimates}
Below is an interesting kernel estimate:

\begin{lemma}\label{KE:ESle}
Given~$R>0$, let\footnote{As customary, we use the notation
$$ \chi_\Omega(x):=\begin{cases} 1 &{\mbox{ if }} x\in\Omega,\\
0 &{\mbox{ if }} x\in\R^n\setminus\Omega.
\end{cases} $$}
\begin{eqnarray*}
&& K_{1,R}(y):=
\chi_{B_{2R}}(y)\int_{\R^n}
\min\left\{1,\frac{|x-y|^2}{R^2}\right\}\,\frac{dx}{|x-y|^{n+2s}}\\ {\mbox{and }}&&K_{2,R}(y):=
\chi_{\R^n\setminus B_{2R}}(y)\int_{B_{2R}}\min\left\{1,\frac{|x-y|^2}{R^2}\right\}\,\frac{dx}{|x-y|^{n+2s}}.
\end{eqnarray*}

Let also~$K_R:=K_{1,R}+K_{2,R}$. Then, there exists~$C>0$, depending only on~$n$ and~$s$, such that
\begin{equation}\label{La678:jkokmv502-y}
K_R(y)\le\frac{CR^n}{(R+|y|)^{n+2s}}.\end{equation}
\end{lemma}

\begin{proof} Let~$y\in B_{4R}$. Then,
\begin{equation}\label{BAL:1} R\ge\frac{R+|y|}5\end{equation}
and
\begin{equation}\label{BAL:2}
\frac1{R^2}\int_{B_{4R}}
\frac{dx}{|x-y|^{n+2s-2}}\le\frac{C}{R^{2s}},
\end{equation}
for some~$C>0$ depending only on~$n$ and~$s$, which we will freely rename in what follows.

Also, if~$y\in B_{2R}$,
\begin{equation}\label{BAL:3}
\int_{\R^n\setminus B_{4R}}
\frac{dx}{|x-y|^{n+2s}}\le\frac{C}{R^{2s}}.
\end{equation}

As a result, combining~\eqref{BAL:1}, \eqref{BAL:2}, and~\eqref{BAL:3},
\begin{equation}\label{vqrsfwygejvoPKS:KP:D31:1}
\begin{split}&K_{1,R}(y)\le
\chi_{B_{2R}}(y)\left( \frac1{R^2}\int_{B_{4R}}
\frac{dx}{|x-y|^{n+2s-2}}+
\int_{\R^n\setminus B_{4R}}
\frac{dx}{|x-y|^{n+2s}}\right)\\&\qquad\qquad\qquad\qquad
\le
\frac{C\chi_{B_{2R}}(y)}{R^{2s}}\le\frac{CR^n}{(R+|y|)^{n+2s}}.
\end{split}\end{equation}

Moreover, if~$x\in B_{2R}$ and~$y\in\R^n\setminus B_{4R}$, we have that
$$|x-y|\ge |y|-|x|\ge|y|-2R\ge\frac{R+|y|}3.$$
On this account, using again~\eqref{BAL:2} we conclude that
\begin{eqnarray*}
K_{2,R}(y)&\le&
\frac{\chi_{B_{4R}\setminus B_{2R}}(y)}{R^2}\int_{B_{2R}}\frac{dx}{|x-y|^{n+2s-2}}+
\chi_{\R^n\setminus B_{4R}}(y)\int_{B_{2R}}\frac{dx}{|x-y|^{n+2s}}\\&\le&\frac{C\chi_{B_{4R}\setminus B_{2R}}(y)}{R^{2s}}+
\frac{CR^n\chi_{\R^n\setminus B_{4R}}(y)}{(R+|y|)^{n+2s}}
\end{eqnarray*}
and thus, using again~\eqref{BAL:1},
$$K_{2,R}(y)\le
\frac{CR^n}{(R+|y|)^{n+2s}}.$$
From this and~\eqref{vqrsfwygejvoPKS:KP:D31:1}, the desired result plainly follows.
\end{proof}

Now we consider a cutoff function.
Given~$R>0$, $f:\R^n\to\R$, and~$\tau\in C^\infty_c(B_2,[0,1])$ with~$\tau=1$ in~$B_1$, we define~$\tau_R(x):=\tau\left(\frac{x}R\right)$
and
\begin{equation}\label{LASJMLv53647r194S}
f_R(x):=f(x)\tau_R(x).\end{equation} 
The link with the kernel estimate in Lemma~\ref{KE:ESle} is provided by the following observation:

\begin{lemma}
In the notation of Lemma~\ref{KE:ESle}, we have that
\begin{equation}\label{La678:jkokmv502-1}
\iint_{\R^{2n}}\frac{f^2(y)(\tau_R(x)-\tau_R(y))^2}{|x-y|^{n+2s}}\,dx\,dy\le C\int_{\R^n}
f^2(y)\,K_R(y)\,dy.
\end{equation}
Also, 
\begin{equation}\label{La678:jkokmv502-1a}
\iint_{\R^{2n}}\frac{f^2(y)(\tau_R(x)-\tau_R(y))^2}{|x-y|^{n+2s}}\,dx\,dy\le CR^n\int_{\R^n}
\frac{f^2(y)}{(R+|y|)^{n+2s}}\,dy,\end{equation}
and, for all~$p\in[2,+\infty]$,
\begin{equation}\label{La678:jkokmv502-2}
\iint_{\R^{2n}}\frac{f^2(y)(\tau_R(x)-\tau_R(y))^2}{|x-y|^{n+2s}}\,dx\,dy\le CR^{\frac{n(p-2)}{p}-2s}\|f\|_{L^p(\R^n)}^2.
\end{equation}
\end{lemma}

\begin{proof} We use the notation
\begin{equation*}
\iint_{A\times B}\textcircled{$g$}=\iint_{A\times B}\frac{f^2(y)(g(x)-g(y))^2}{|x-y|^{n+2s}}\,dx\,dy
.\end{equation*}
We observe that if~$x$, $y\in B_R$ then~$\tau_R(x)-\tau_R(y)=1-1=0$.
Similarly, if~$x$, $y\in\R^n\setminus B_{2R}$ then~$\tau_R(x)-\tau_R(y)=0-0=0$.
Therefore,
\begin{equation}\begin{split}\label{La678:jkokmv502-x}&
\frac12\iint_{\R^n\times \R^n}\textcircled{{$\tau_R$}}\le
\iint_{\R^n\times B_{2R}}\textcircled{{$\tau_R$}}+\iint_{\R^n\times(\R^n\setminus B_{2R})}\textcircled{{$\tau_R$}}
\\&\quad=\iint_{\R^n\times B_{2R}}\textcircled{{$\tau_R$}}+ \iint_{ B_{2R}\times (\R^n\setminus B_{2R})}\textcircled{{$\tau_R$}}\\&\quad=
\iint_{\R^n\times\R^n} \chi_{B_{2R}}(y)
\frac{f^2(y)(\tau_R(x)-\tau_R(y))^2}{|x-y|^{n+2s}}\,dx\,dy\\&\qquad\qquad+
\iint_{B_{2R}\times\R^n} \chi_{\R^n\setminus B_{2R}}(y)
\frac{f^2(y)(\tau_R(x)-\tau_R(y))^2}{|x-y|^{n+2s}}\,dx\,dy.
\end{split}\end{equation}
Also, since~$\|\nabla\tau_R\|_{L^\infty(\R^n)}\le\frac{C}{R}$, we have that
$$|\tau_R(x)-\tau_R(y)|\le C\min\left\{1,\frac{|x-y|}{R}\right\}.$$
This observation, \eqref{La678:jkokmv502-x}, and the definition of~$K_R$ in Lemma~\ref{KE:ESle} yield the desired result in~\eqref{La678:jkokmv502-1}.

From~\eqref{La678:jkokmv502-y} and~\eqref{La678:jkokmv502-1}, the claim
in~\eqref{La678:jkokmv502-1a} plainly follows.

Now, if~$p\in(2,+\infty)$ we use~\eqref{La678:jkokmv502-1a} in combination with the H\"older Inequality with conjugate exponents~$\frac{p}{2}$ and~$\frac{p}{p-2}$, finding that
\begin{eqnarray*}&&
\iint_{\R^{2n}}\frac{f^2(y)(\tau_R(x)-\tau_R(y))^2}{|x-y|^{n+2s}}\,dx\,dy\le CR^n\|f\|^2_{L^p(\R^n)}\left(\int_{\R^n}
\frac{dy}{(R+|y|)^{\frac{(n+2s)p}{p-2}}}\right)^{\frac{p-2}{p}}.
\end{eqnarray*}
The change of variable~$z:=\frac{y}R$ yields~\eqref{La678:jkokmv502-2} in this case. 

The cases~$p=2$ and~$p++\infty$ follow directly from~\eqref{La678:jkokmv502-1a}.
Thus~\eqref{La678:jkokmv502-2} holds for all~$p\in[2,+\infty]$, as advertised.
\end{proof}

\begin{corollary}\label{f0dojflergh4-5ipyhk3} Let~$p\in[2,2^*_s]$. Assume that~$f\in L^p(\R^n)$ and
\begin{equation}\label{TCmapmin}
\iint_{\R^{2n}}\frac{|f(x)-f(y)|^2}{|x-y|^{n+2s}}\,dx\,dy<+\infty.\end{equation}

Then, in the notation of~\eqref{LASJMLv53647r194S},
\begin{equation*}
\lim_{R\to+\infty} \iint_{\R^{2n}}\frac{((f-f_R)(x)-(f-f_R)(y))^2}{|x-y|^{n+2s}}\,dx\,dy=0.
\end{equation*}
\end{corollary}

\begin{proof} The proof is easier when~$p\in[2,2^*_s)$, but in this paper the
crucial exponent is precisely~$p=2^*_s$, so we provide a comprehensive proof.
Given~$\e>0$, we pick~$f^\star_\e\in C^\infty_c(\R^n)$ such that~$\|f-f^\star_\e\|_{L^p(\R^n)}<\epsilon$.

We also observe that
\begin{equation}\label{BNn5fd02fgdhs0hfwty-1}\begin{split}&
\big|(f-f_R)(x)-(f-f_R)(y)\big|=
\big|f(x)(1-\tau_R)(x)-f(y)(1-\tau_R)(y)\big|\\&\qquad\le
|f(x)-f(y)|\,(1-\tau_R(x))+|f(y)|\,|\tau_R(x)-\tau_R(y)|\\&\qquad\le
|f(x)-f(y)|\,(1-\tau_R(x))+|f^\star_\e(y)|\,|\tau_R(x)-\tau_R(y)|\\&\qquad\qquad+|(f-f^\star_\e)(y)|\,|\tau_R(x)-\tau_R(y)|.\end{split}
\end{equation}

Now, by~\eqref{TCmapmin} and the Dominated Convergence Theorem,
\begin{equation}\label{BNn5fd02fgdhs0hfwty-2}\lim_{R\to+\infty}
\iint_{\R^{2n}}\frac{|f(x)-f(y)|^2\,(1-\tau_R(x))}{|x-y|^{n+2s}}\,dx\,dy
=0.
\end{equation}

Moreover, using~\eqref{La678:jkokmv502-2} with, e.g., $p=2$,
\begin{equation}\label{BNn5fd02fgdhs0hfwty-3}\lim_{R\to+\infty}
\iint_{\R^{2n}}\frac{|f^\star_\e(y)|^2 (\tau_R(x)-\tau_R(y))^2}{|x-y|^{n+2s}}\,dx\,dy\le
C\lim_{R\to+\infty} R^{-2s}\|f_\e^\star\|_{L^2(\R^n)}^2=0.
\end{equation}

Besides, it follows from our assumption on~$p$ that~${\frac{n(p-2)}{p}-2s}\le0$ and consequently, by~\eqref{La678:jkokmv502-2},
\begin{equation*}\begin{split}&\lim_{R\to+\infty}
\iint_{\R^{2n}}\frac{
|(f-f^\star_\e)(y)|^2
 (\tau_R(x)-\tau_R(y))^2}{|x-y|^{n+2s}}\,dx\,dy\\&\quad\le C\|f-f_\e^\star\|_{L^p(\R^n)}^2
\lim_{R\to+\infty} 
 R^{\frac{n(p-2)}{p}-2s}\le C\|f-f_\e^\star\|_{L^p(\R^n)}^2\le C\epsilon^2.
\end{split}\end{equation*}

Putting together this observation, \eqref{BNn5fd02fgdhs0hfwty-1}, \eqref{BNn5fd02fgdhs0hfwty-2}, and~\eqref{BNn5fd02fgdhs0hfwty-3}, and using that~$(a+b+c)^2\le 4(a^2+b^2+c^2)$, we gather that
$$\lim_{R\to+\infty} \iint_{\R^{2n}}\frac{((f-f_R)(x)-(f-f_R)(y))^2}{|x-y|^{n+2s}}\,dx\,dy\le C\epsilon^2,$$thus the desired result follows by taking~$\e$ as small as we wish.
\end{proof}

Now we let
\begin{equation}\label{XPMAODe} X_p:=\left\{f\in{L^{p}(\R^n)}{\mbox{ s.t. }}
\iint_{\R^{2n}}\frac{(f(x)-f(y))^2}{|x-y|^{n+2s}}\,dx\,dy<+\infty
\right\},
\end{equation}
endowed with its natural norm
$$\|f\|_{X_p}:= \|f\|_{L^{p}(\R^n)}+\sqrt{\iint_{\R^{2n}}\frac{(f(x)-f(y))^2}{|x-y|^{n+2s}}\,dx\,dy}.$$
The preparatory work above allows us to obtain the following useful density
property.

\begin{proposition}\label{prop:dens}
Let~$p\in[2,2^*_s]$. Then,
smooth and compactly supported functions are dense
in~$X_p$.
\end{proposition}

\begin{proof} 
Let $$ H^s(\R^n):=\left\{f\in{L^{2}(\R^n)}{\mbox{ s.t. }}
\iint_{\R^{2n}}\frac{(u(x)-u(y))^2}{|x-y|^{n+2s}}\,dx\,dy<+\infty
\right\},$$
endowed with its natural norm.

By the fractional Sobolev Inequality (see e.g.~\cite[Theorems~2.4 and~6.5]{MR2944369}), we know that, for all~$v\in H^s(\R^n)$,
\begin{equation}\label{Hnasla-02}\|v\|_{L^{2^*_s}(\R^n)}^2\le C_{n,s}\iint_{\R^{2n}}\frac{(v(x)-v(y))^2}{|x-y|^{n+2s}}\,dx\,dy,\end{equation}
for some~$C_{n,s}>0$.

Let~$f\in X_p$. Recalling the notation of~\eqref{LASJMLv53647r194S}, we know that
$$ \lim_{R\to+\infty}\|f-f_R\|_{L^p(\R^n)}=0.$$
Thus, given~$\e>0$, we combine this observation with Corollary~\ref{f0dojflergh4-5ipyhk3} to pick~$R_\e$ such that~$\|f-f_{R_\e}\|_{X_p}\le\e$. Since~$f_{R_\e}$
is compactly supported, we have that
\begin{equation}\label{A:913uij-210Kj3ik}
f_{R_\e}\in H^s(\R^n).\end{equation}

Since, in view of~\cite[Theorem~2.4]{MR2944369}, we know that
smooth and compactly supported functions are dense
in~$ H^s(\R^n)$, we can find~$g_{\e}\in C^\infty_c(\R^n)$ such that~$\|f_{R_\e}-g_{\e}\|_{H^s(\R^n)}\le\e$.

Furthermore,
if~$p\in(2,2^*_s)$, using the H\"older Inequality with exponents~$\frac{2^*_s-2}{p-2}$ and~$\frac{2^*_s-2}{2^*_s-p}$, for any function~$h:\R^n\to\R$ we have that
$$ \|h\|_{L^p(\R^n)}=
\left(\int_{\R^n}|h(x)|^{\frac{2^*_s(p-2)}{2^*_s-2}}|h(x)|^{
\frac{2(2^*_s-p)}{2^*_s-2}
}\,dx \right)^{\frac1p}
\le \|h\|_{L^{2^*_s}(\R^n)}^{\frac{2^*_s(p-2)}{p(2^*_s-2)}} \|h\|_{L^2(\R^n)}^{\frac{2(2^*_s-p)}{p(2^*_s-2)}}.$$
Notice that for~$p=2$ and~$p=2^*_s$ the same estimate trivially holds true.

As a result, 
\begin{eqnarray*}
&&\|f-g_{\e}\|_{X_p}\le\e+\|f_{R_\e}-g_\e\|_{X_p}\le2\e+
\|f_{R_\e}-g_\e\|_{L^p(\R^n)}\\&&\qquad\le2\e+
\|f_{R_\e}-g_\e\|_{L^{2^*_s}(\R^n)}^{\frac{2^*_s(p-2)}{p(2^*_s-2)}} \|f_{R_\e}-g_\e\|_{L^2(\R^n)}^{\frac{2(2^*_s-p)}{p(2^*_s-2)}}\le
2\e+\|f_{R_\e}-g_\e\|_{L^{2^*_s}(\R^n)}^{\frac{2^*_s(p-2)}{p(2^*_s-2)}}
\|f_{R_\e}-g_\e\|_{H^s(\R^n)}^{\frac{2(2^*_s-p)}{p(2^*_s-2)}}
.\end{eqnarray*}

Notice also that~\eqref{A:913uij-210Kj3ik} allows us to use~\eqref{Hnasla-02}
with~$v:=f_{R_\e}-g_\e$, yielding (up to renaming~$C_{n,s}$) that
\begin{eqnarray*}
&&\|f-g_{\e}\|_{X_p}\le
2\e+C_{n,s}\|f_{R_\e}-g_\e\|_{H^s(\R^n)}^{\frac{2^*_s(p-2)}{p(2^*_s-2)}}
\|f_{R_\e}-g_\e\|_{H^s(\R^n)}^{\frac{2(2^*_s-p)}{p(2^*_s-2)}}\\&&\;\quad=
2\e+C_{n,s}\|f_{R_\e}-g_\e\|_{H^s(\R^n)}\le (2+C_{n,s})\e,
\end{eqnarray*}which establishes the desired result.
\end{proof}

\subsection{Functional inequalities}
Here, we recall a fractional Sobolev Inequality:

\begin{lemma}\label{SOBLEE} We have that~$
{\mathfrak{S}}_s>0$.
\end{lemma}

\begin{proof} Let ${u\in L^{2^*_s}(\R^n_+)}$ with ${{\|u\|_{L^{2^*_s}(\R^n_+)}=1}}$. We claim that
\begin{equation}\label{Hnasla-01}
\iint_{Q(\R^{n}_+)}\frac{(u(x)-u(y))^2}{|x-y|^{n+2s}}\,dx\,dy\ge
\frac{2^{\frac{2}{2^*_s}-2}}{C_{n,s}}.
\end{equation}
To prove this, we can assume that the term on the left-hand side of~\eqref{Hnasla-01}
is finite, otherwise we are done.
We define
$$v(x',x_n):=\begin{dcases}
u(x',x_n)&{\mbox{ if }}x_n\ge0,\\
u(x',-x_n)&{\mbox{ if }}x_n<0,
\end{dcases}$$and we stress that~$\|v\|_{L^{2^*_s}(\R^n)}^{2^*_s}=2\|u\|_{L^{2^*_s}(\R^n_+)}^{2^*_s}=2$.

In addition, if~$x$, $z\in\R^n_+$, then
$$  |(x',x_n)-(z',-z_n)|^2-|x-z|^2=(x_n+z_n)^2-(x_n-z_n)^2=4x_nz_n\ge0,$$
therefore
\begin{eqnarray*}&&\!\!\!\!\!\!\!
\frac12\iint_{\R^{2n}}\frac{(v(x)-v(y))^2}{|x-y|^{n+2s}}\,dx\,dy
\\&&=\iint_{\R^n_+\times\R^n_+}\frac{(u(x)-u(y))^2}{|x-y|^{n+2s}}\,dx\,dy+
\iint_{\R^n_+\times\R^n_-}\frac{(u(x',x_n)-u(y',-y_n))^2}{|x-y|^{n+2s}}\,dx\,dy\\&&=\iint_{\R^n_+\times\R^n_+}\frac{(u(x)-u(y))^2}{|x-y|^{n+2s}}\,dx\,dy+
\iint_{\R^n_+\times\R^n_+}\frac{(u(x)-u(z))^2}{|(x',x_n)-(z',-z_n)|^{n+2s}}\, dx\,dz\\&&\le2\iint_{\R^n_+\times\R^n_+}\frac{(u(x)-u(y))^2}{|x-y|^{n+2s}}\,dx\,dy
\\&&\le2\iint_{Q(\R^n_+)}\frac{(u(x)-u(y))^2}{|x-y|^{n+2s}}\,dx\,dy,
\end{eqnarray*}which is finite.

As a result, in the notation of~\eqref{XPMAODe},
we have that~$v\in X_{2^*_s}$ and thus, in light of Proposition~\ref{prop:dens}, given~$\e>0$,
we can find~$v_\e\in C^\infty_c(\R^n)$ with~$\|v_\e-v\|_{X_{2^*_s}}\le\e$.

We can also use the fractional Sobolev Inequality in~\eqref{Hnasla-02} to see that
\begin{equation*}\|v_\e\|_{L^{2^*_s}(\R^n)}^2\le C_{n,s}\iint_{\R^{2n}}\frac{(v_\e(x)-v_\e(y))^2}{|x-y|^{n+2s}}\,dx\,dy.\end{equation*}

All in all,
\begin{eqnarray*}&&\sqrt{
\iint_{Q(\R^n_+)}\frac{(u(x)-u(y))^2}{|x-y|^{n+2s}}\,dx\,dy}
\ge\frac12\sqrt{\iint_{\R^{2n}}\frac{(v(x)-v(y))^2}{|x-y|^{n+2s}}\,dx\,dy}\\
&&\quad\ge\frac12
\left(\sqrt{\iint_{\R^{2n}}\frac{(v_\e(x)-v_\e(y))^2}{|x-y|^{n+2s}}\,dx\,dy}
-\|v_\e-v\|_{X_{2^*_s}}\right)
\\&&\quad\ge\frac12\left(\frac1{\sqrt{C_{n,s}}}\|v_\e\|_{L^{2^*_s}(\R^n)}
-\|v_\e-v\|_{X_{2^*_s}}\right)\\&&\quad\ge\frac12\left(\frac1{\sqrt{C_{n,s}}}\|v\|_{L^{2^*_s}(\R^n)}-\left(\frac1{\sqrt{C_{n,s}}}+1\right)\|v_\e-v\|_{X_{2^*_s}}\right)\\&&\quad\ge\frac12\left(\frac{2^{\frac1{2^*_s}}}{\sqrt{C_{n,s}}}\|u\|_{L^{2^*_s}(\R^n_+)}-\left(\frac1{\sqrt{C_{n,s}}}+1\right)\e\right)\\&&\quad=\frac12\left(\frac{2^{\frac1{2^*_s}}}{\sqrt{C_{n,s}}}-\left(\frac1{\sqrt{C_{n,s}}}+1\right)\e\right).
\end{eqnarray*}
Hence, taking~$\e$ arbitrarily small,
$$\sqrt{
\iint_{Q(\R^n_+)}\frac{(u(x)-u(y))^2}{|x-y|^{n+2s}}\,dx\,dy}
\ge\frac{2^{\frac1{2^*_s}-1}}{\sqrt{C_{n,s}}}$$
which establishes the claim in~\eqref{Hnasla-01}.

The desired result then follows by taking the infimum in~$u$ in~\eqref{Hnasla-01}.
\end{proof}

\end{appendix}

%\bibliography{system}
%\bibliographystyle{abbrv}

\vspace{0.5cm}

\end{document}